\documentclass[a4paper,10pt,reqno]{amsart}
\usepackage[latin1]{inputenc}
\usepackage{dsfont}
\usepackage{amsxtra}
\usepackage{amsmath}
\usepackage{amssymb}
\usepackage{amsfonts}
\usepackage[all]{xy}
\usepackage{mathrsfs}
\usepackage{amsthm}

\theoremstyle{plain}
\newtheorem{thm}{Theorem}[section]
\newtheorem{lemme}[thm]{Lemma}
\newtheorem{prop}[thm]{Proposition}
\newtheorem{cor}[thm]{Corollary}

\theoremstyle{definition}
\newtheorem{df}[thm]{Definition}

\theoremstyle{remark}
\newtheorem{rmq}[thm]{Remark}

\numberwithin{equation}{section}

\newcommand{\C}{\mathbb{C}}
\newcommand{\Z}{\mathbb{Z}}
\newcommand{\N}{\mathbb{N}}

\newcommand{\calA}{\mathcal{A}}

\newcommand{\calC}{\mathcal{C}}
\newcommand{\calD}{\mathcal{D}}

\newcommand{\calF}{\mathcal{F}}

\newcommand{\calO}{\mathcal{O}}
\newcommand{\calP}{\mathcal{P}}

\newcommand{\calS}{\mathcal{S}}

\newcommand{\bfh}{\mathbf{h}}
\newcommand{\bfq}{\mathbf{q}}
\newcommand{\bfs}{\mathbf{s}}
\newcommand{\bfz}{\mathbf{z}}
\newcommand{\bfu}{\mathbf{u}}

\newcommand{\Lie}{\mathfrak}

\newcommand{\h}{\mathfrak{h}}

\newcommand{\al}{\alpha}

\newcommand{\vep}{\varepsilon}
\newcommand{\lam}{\lambda}

\newcommand{\sss}{\scriptscriptstyle}

\newcommand{\ra}{\rightarrow}
\newcommand{\leqs}{\leqslant}
\newcommand{\geqs}{\geqslant}
\newcommand{\lra}{\longrightarrow}
\newcommand{\simra}{\overset{\sim}\ra}
\newcommand{\simlra}{\overset{\sim}\lra}

\newcommand{\Hecke}{\mathscr{H}}
\newcommand{\pair}[1]{\langle{#1}\rangle}
\newcommand{\sle}{\widehat{\Lie{sl}}_e}
\newcommand{\Resh}{\sideset{^{\sss\Hecke}}{^W_{W'}}\Res}
\newcommand{\Indh}{\sideset{^{\sss\Hecke}}{^W_{W'}}\Ind}
\newcommand{\coIndh}{\sideset{^{\sss\Hecke}}{^W_{W'}}\coInd}

\DeclareMathOperator{\Hom}{\mathrm{Hom}}

\DeclareMathOperator{\Ind}{\mathrm {Ind}}
\DeclareMathOperator{\coInd}{\mathrm {coInd}}

\DeclareMathOperator{\Res}{\mathrm {Res}}
\DeclareMathOperator{\res}{\mathrm {res}}
\DeclareMathOperator{\soc}{\mathrm {soc}}
\DeclareMathOperator{\head}{\mathrm {head}}

\DeclareMathOperator{\op}{\mathrm{op}}
\DeclareMathOperator{\Id}{\mathrm{Id}}
\DeclareMathOperator{\pr}{\mathrm{pr}}

\DeclareMathOperator{\Proj}{\mathrm{Proj}}

\DeclareMathOperator{\End}{\mathrm{End}}
\DeclareMathOperator{\Fct}{\mathrm{Fct}}

\DeclareMathOperator{\wt}{\mathrm{wt}}

\DeclareMathOperator{\KZ}{\mathrm{KZ}}

\DeclareMathOperator{\eu}{\mathbf{eu}}
\DeclareMathOperator{\modu}{\mathrm{-mod}}

\DeclareMathOperator{\Irr}{\mathrm{Irr}}

\author{Peng Shan}
\address{Universit{\'e} Paris 7, Institut de Math{\'e}matiques de
Jussieu, Th{\'e}orie des groupes et des repr{\'e}sentations, Case
7012, 2 place Jussieu, 75251 Paris Cedex 05, France.}
\email{shan@math.jussieu.fr}

\title[Crystals and DAHA]{Crystals of Fock spaces and cyclotomic rational double affine Hecke algebras}

\thanks{\emph{Titre en fran\c{c}ais.} Cristaux
d'espaces de Fock et alg\`ebres de Hecke doublement affine
rationnelles cyclotomiques.}

\begin{document}

\begin{abstract}

We define the $i$-restriction and $i$-induction functors on the
category $\calO$ of the cyclotomic rational double affine Hecke
algebras. This yields a crystal on the set of isomorphism classes of
simple modules, which is isomorphic to the crystal of a Fock space.
\\
\\
\noindent\textsc{R\'esum\'e.} On d\'efinit les foncteurs de
$i$-restriction et $i$-induction sur la cat\'egorie $\calO$ des
alg\`ebres de Hecke doublement affine rationnelles cyclotomiques.
Ceci donne lieu \`a un cristal sur l'ensemble des classes
d'isomorphismes de modules simples, qui est isomorphe au cristal
d'un espace de Fock.

\end{abstract}

\maketitle

\section*{Introduction}

In \cite{A1}, S. Ariki defined the $i$-restriction and $i$-induction
functors for cyclotomic Hecke algebras. He showed that the
Grothendieck group of the category of finitely generated projective
modules of these algebras admits a module structure over the affine
Lie algebra of type $A^{(1)}$, with the action of Chevalley
generators given by the $i$-restriction and $i$-induction functors.

The restriction and induction functors for rational DAHA's(=double
affine Hecke algebras) were recently defined by R. Bezrukavnikov and
P. Etingof. With these functors, we give an analogue of Ariki's
construction for the category $\calO$ of cyclotomic rational DAHA's:
we show that as a module over the type $A^{(1)}$ affine Lie algebra,
the Grothendieck group of this category is isomorphic to a Fock
space. We also construct a crystal on the set of isomorphism classes
of simple modules in the category $\calO$. It is isomorphic to the
crystal of the Fock space. Recall that this Fock space also enters
in some conjectural description of the decomposition numbers for the
category $\calO$ considered here. See \cite{U}, \cite{Y}, \cite{R}
for related works.

\section*{Notation}

For $A$ an algebra, we will write $A\modu$ for the category of
finitely generated $A$-modules. For $f: A\ra B$ an algebra
homomorphism from $A$ to another algebra $B$ such that $B$ is
finitely generated over $A$, we will write
$$f_\ast: B\modu\ra A\modu$$ for the restriction functor and we write $$f^\ast: A\modu\ra B\modu,\quad M\mapsto B\otimes_AM.$$

A $\C$-linear category $\calA$ is called artinian if the Hom sets
are finite dimensional $\C$-vector spaces and every object has a
finite length. Given an object $M$ in $\calA$, we denote by
$\soc(M)$ (resp. $\head(M)$) the socle (resp. the head) of $M$,
which is the largest semi-simple subobject (quotient) of $M$.

Let $\calC$ be an abelian category. The Grothendieck group of
$\calC$ is the quotient of the free abelian group generated by
objects in $\mathcal{C}$ modulo the relations $M=M'+M''$ for all
objects $M,M',M''$ in $\mathcal{C}$ such that there is an exact
sequence $0\ra M'\ra M\ra M''\ra 0$. Let $K(\mathcal{C})$ denote the
complexified Grothendieck group, a $\C$-vector space. For each
object $M$ in $\mathcal{C}$, let $[M]$ be its class in
$K(\mathcal{C})$. Any exact functor $F: \mathcal{C}\ra\mathcal{C}'$
between two abelian categories induces a vector space homomorphism
$K(\mathcal{C})\ra K(\mathcal{C}')$, which we will denote by $F$
again. Given an algebra $A$ we will abbreviate $K(A)=K(A\modu)$.

Denote by $\Fct(\mathcal{C},\mathcal{C}')$ the category of functors
from a category $\mathcal{C}$ to a category $\mathcal{C}'$. For
$F\in\Fct(\mathcal{C},\mathcal{C}')$ write $\End(F)$ for the ring of
endomorphisms of the functor $F$. We denote by $1_F: F\ra F$ the
identity element in $\End(F)$. Let
$G\in\Fct(\mathcal{C'},\mathcal{C''})$ be a functor from
$\mathcal{C}'$ to another category $\mathcal{C}''$. For any
$X\in\End(F)$ and any $X'\in\End(G)$ we write $X'X:G\circ F\ra
G\circ F$ for the morphism of functors given by
$X'X(M)=X'(F(M))\circ G(X(M))$ for any $M\in\mathcal{C}$.

Let $e\geqs 2$ be an integer and $z$ be a formal parameter. Denote
by $\Lie{sl}_e$ the Lie algebra of traceless $e\times e$ complex
matrices. Write $E_{ij}$ for the elementary matrix with $1$ in the
position $(i,j)$ and $0$ elsewhere. The type $A^{(1)}$ affine Lie
algebra $\widehat{\Lie{sl}}_e$ is
$\Lie{sl}_e\otimes\C[z,z^{-1}]\oplus\C c$ with $c$ a central
element. The Lie bracket is the usual one. We will denote the
Chevalley generators of $\widehat{\Lie{sl}}_e$ as follows:
\begin{eqnarray*}
&e_i=E_{i,i+1}\otimes 1,\quad &f_i=F_{i+1,i}\otimes 1,\quad
h_i=(E_{ii}-E_{i+1,i+1})\otimes 1, \quad 1\leqs i\leqs e-1,\\
&e_0=E_{e1}\otimes z,\quad &f_0=E_{1e}\otimes z^{-1},\quad
h_0=(E_{ee}-E_{11})\otimes 1+c.
\end{eqnarray*}
For $i\in\Z/e\Z$ we will denote the simple root (resp. coroot)
corresponding to $e_i$ by $\al_i$ (resp. $\al\spcheck_i$). The
fundamental weights are $\{\Lambda_i: i\in\Z/e\Z\}$ with
$\al\spcheck_i(\Lambda_j)=\delta_{ij}$ for any $i,j\in\Z/e\Z$. We
will write $P$ for the weight lattice, the free abelian group
generated by the fundamental weights.

\section{Reminders on Hecke algebras, rational DAHA's and restriction
functors}\label{s:reminder}

\subsection{Hecke algebras.}\label{ss:Hecke}

Let $\h$ be a finite dimensional vector space over $\C$. Recall that
a pseudo-reflection is a non trivial element $s$ of $GL(\h)$ which
acts trivially on a hyperplane, called the reflecting hyperplane of
$s$. Let $W\subset GL(\h)$ be a finite subgroup generated by
pseudo-reflections. Let $\mathcal{S}$ be the set of
pseudo-reflections in $W$ and $\mathcal{A}$ be the set of reflecting
hyperplanes. We set $\h_{reg}=\h-\bigcup_{H\in\mathcal{A}}H$, it is
stable under the action of $W$. Fix $x_0\in \h_{reg}$ and identify
it with its image in $\h_{reg}/W$. By definition the braid group
attached to $(W,\h)$, denoted by $B(W,\h)$, is the fundamental group
$\pi_1(\h_{reg}/W, x_0).$

For any $H\in\mathcal{A}$, let $W_H$ be the pointwise stabilizer of
$H$. This is a cyclic group. Write $e_H$ for the order of $W_H$. Let
$s_H$ be the unique element in $W_H$ whose determinant is
$\exp(\frac{2\pi\sqrt{-1}}{e_H})$.
Let $q$ be a map from $\mathcal{S}$ to $\C^\ast$ that is constant on
the $W$-conjugacy classes. Following \cite[Definition 4.21]{BMR} the
Hecke algebra $\Hecke_q(W,\h)$ attached to $(W,\h)$ with parameter
$q$ is the quotient of the group algebra $\C B(W,\h)$ by the
relations:
\begin{equation}\label{heckerelation}
(T_{s_H}-1)\prod_{t\in W_H\cap\mathcal{S}}(T_{s_H}-q(t))=0,\quad
H\in\mathcal{A}.
\end{equation}
Here $T_{s_H}$ is a generator of the monodromy around $H$ in
$\h_{reg}/W$ such that the lift of $T_{s_H}$ in $\pi_1(W,\h_{reg})$
via the map $\h_{reg}\ra\h_{reg}/W$ is represented by a path from
$x_0$ to $s_H(x_0)$. See \cite[Section 2B]{BMR} for a precise
definition. When the subspace $\h^W$ of fixed points of $W$ in $\h$
is trivial, we abbreviate
$$B_W=B(W,\h), \quad \Hecke_q(W)=\Hecke_q(W,\h).$$

\subsection{Parabolic restriction and induction for Hecke algebras.}\label{ss:resHecke}

In this section we will assume that $\h^W=1$. A parabolic subgroup
$W'$ of $W$ is by definition the stabilizer of a point $b\in\h$. By
a theorem of Steinberg, the group $W'$ is also generated by
pseudo-reflections. Let $q'$ be the restriction of $q$ to
$\mathcal{S'}=W'\cap \mathcal{S}$. There is an explicit inclusion
$\imath_q: \Hecke_{q'}(W')\hookrightarrow \Hecke_q(W)$ given by
\cite[Section 2D]{BMR}. The restriction functor
\begin{equation*}
\Resh:\Hecke_q(W)\modu\ra\Hecke_{q'}(W')\modu
\end{equation*} is the functor
$(\imath_q)_\ast$. The induction functor
$$\Indh=\Hecke_q(W)\otimes_{\Hecke_{q'}(W')}-$$
is left adjoint to $\Resh$. The coinduction functor
$$\coIndh=\Hom_{\Hecke_{q'}(W')}(\Hecke_q(W),-)$$ is right adjoint to
$\Resh$. The three functors above are all exact.

Let us recall the definition of $\imath_q$. It is induced from an
inclusion $\imath: B_{W'}\hookrightarrow B_{W}$, which is in turn
the composition of three morphisms $\ell$, $\kappa$, $\jmath$
defined as follows. First, let $\mathcal{A}'\subset\mathcal{A}$ be
the set of reflecting hyperplanes of $W'$. Write
$$\overline{\h}=\h/\h^{W'},\quad\overline{\mathcal{A}}=\{\overline{H}=H/\h^{W'}:\,H\in
\mathcal{A}'\}, \quad
\overline{\h}_{reg}=\overline{\h}-\bigcup_{\overline{H}\in\overline{\mathcal{A}}}\overline{H},
\quad\h'_{reg}=\h-\bigcup_{H\in\mathcal{A}'}H.$$ The canonical
epimorphism $p: \h\ra\overline{\h}$ induces a trivial
$W'$-equivariant fibration $p: \h'_{reg}\ra \overline{\h}_{reg}$,
which yields an isomorphism
\begin{equation}\label{heckeres1}
\ell: B_{W'}=\pi_1(\overline{\h}_{reg}/{W'},
p(x_0))\overset{\sim}\ra \pi_1(\h'_{reg}/W',x_0).
\end{equation}

Endow $\h$ with a $W$-invariant hermitian scalar product. Let
$||\cdot||$ be the associated norm. Set
\begin{equation}\label{eq:omega}
\Omega=\{x\in\h:\,||x-b||< \varepsilon\},
\end{equation}
where $\varepsilon$ is a positive real number such that the closure
of $\Omega$ does not intersect any hyperplane that is in the
complement of $\mathcal{A}'$ in $\mathcal{A}$. Let $\gamma: [0,1]\ra
\h$ be a path such that $\gamma(0)=x_0$, $\gamma(1)=b$ and
$\gamma(t)\in\h_{reg}$ for $0<t<1$. Let $u\in[0,1[$ such that
$x_1=\gamma(u)$ belongs to $\Omega$, write $\gamma_u$ for the
restriction of $\gamma$ to $[0,u]$. Consider the homomorphism
\begin{equation*}
\sigma: \pi_1(\Omega\cap\h_{reg},x_1)\ra\pi_1(\h_{reg}, x_0),
\quad\lam\mapsto \gamma^{-1}_u\cdot\lam\cdot\gamma_u.
\end{equation*} The canonical
inclusion $\h_{reg}\hookrightarrow\h'_{reg}$ induces a homomorphism
$\pi_1(\h_{reg}, x_0)\ra \pi_1(\h'_{reg}, x_0)$. Composing it with
$\sigma$ gives an invertible homomorphism
$$\pi_1(\Omega\cap\h_{reg},x_1)\ra\pi_1(\h'_{reg}, x_0).$$ Since
$\Omega$ is $W'$-invariant, its inverse gives an isomorphism
\begin{equation}\label{heckeres2}
\kappa:\pi_1(\h'_{reg}/W',
x_0)\overset{\sim}\ra\pi_1((\Omega\cap\h_{reg})/W',x_1).
\end{equation}

Finally, we see from above that $\sigma$ is injective. So it induces
an inclusion
$$\pi_1((\Omega\cap\h_{reg})/W',x_1)\hookrightarrow\pi_1(\h_{reg}/W', x_0).$$
Composing it with the canonical inclusion $\pi_1(\h_{reg}/W',
x_0)\hookrightarrow \pi_1(\h_{reg}/W, x_0)$ gives an injective
homomorphism
\begin{equation}\label{heckeres3}
\jmath:\pi_1((\Omega\cap\h_{reg})/W',x_1)\hookrightarrow
\pi_1(\h_{reg}/W, x_0)=B_W.
\end{equation}
By composing $\ell$, $\kappa$, $\jmath$ we get the inclusion
\begin{equation}\label{heckeres4}
\imath=\jmath\circ\kappa\circ\ell: B_{W'}\hookrightarrow B_W.
\end{equation}
It is proved in \cite[Section 4C]{BMR} that $\imath$ preserves the
relations in (\ref{heckerelation}). So it induces an inclusion of
Hecke algebras which is the desired inclusion
\begin{equation*}
\imath_q: \Hecke_{q'}(W')\hookrightarrow \Hecke_q(W).
\end{equation*}

For $\imath$, $\imath': B_{W'}\hookrightarrow B_W$ two inclusions
defined as above via different choices of the path $\gamma$, there
exists an element $\rho\in P_W=\pi_1(\h_{reg},x_0)$ such that for
any $a\in B_{W'}$ we have $\imath(a)=\rho\imath'(a)\rho^{-1}$. In
particular, the functors $\imath_\ast$ and $(\imath')_\ast$ from
$B_W\modu$ to $B_{W'}\modu$ are isomorphic. Also, we have
$(\imath_q)_\ast\cong(\imath_q')_\ast.$ So there is a unique
restriction functor $\Resh$ up to isomorphisms.

\subsection{Rational DAHA's.}\label{ss:DAHA}

Let $c$ be a map from $\mathcal{S}$ to $\C$ that is constant on the
$W$-conjugacy classes. The rational DAHA attached to $W$ with
parameter $c$ is the quotient $H_c(W,\h)$ of the smash product of
$\C W$ and the tensor algebra of $\h\oplus\h^\ast$ by the relations
\begin{equation*}
[x,x']=0,\quad[y,y']=0,\quad
[y,x]=\pair{x,y}-\sum_{s\in\mathcal{S}}c_s\pair{\al_s,
y}\pair{x,\al_s\spcheck}s,
\end{equation*}
for all $x,x'\in\h^\ast$, $y,y'\in\h$. Here $\pair{\cdot,\cdot}$ is
the canonical pairing between $\h^\ast$ and $\h$, the element
$\al_s$ is a generator of $\mathrm{Im}(s|_{\h^\ast}-1)$ and
$\al_s\spcheck$ is the generator of $\mathrm{Im}(s|_{\h}-1)$ such
that $\pair{\al_s, \al_s\spcheck}=2$.

For $s\in\mathcal{S}$ write $\lam_s$ for the non trivial eigenvalue
of $s$ in $\h^\ast$. Let $\{x_i\}$ be a basis of $\h^\ast$ and let
$\{y_i\}$ be the dual basis. Let
\begin{equation}\label{euler1}
\mathbf{eu}=\sum_{i}x_iy_i+\frac{\dim(\h)}{2}-\sum_{s\in\mathcal{S}}\frac{2c_s}{1-\lam_s}s
\end{equation}
be the Euler element in $H_c(W,\h)$. Its definition is independent
of the choice of the basis $\{x_i\}$. We have
\begin{equation}\label{euler2}
[\mathbf{eu},x_i]=x_i,\quad [\mathbf{eu},y_i]=-y_i,\quad
[\mathbf{eu},s]=0.
\end{equation}

\subsection{}\label{ss:catO}

The category $\calO$ of $H_c(W,\h)$ is the full subcategory
$\calO_c(W,\h)$ of the category of $H_c(W,\h)$-modules consisting of
objects that are finitely generated as $\C[\h]$-modules and
$\h$-locally nilpotent. We recall from \cite[Section 3]{GGOR} the
following properties of $\calO_c(W,\h)$.

The action of the Euler element $\mathbf{eu}$ on a module in
$\calO_c(W,\h)$ is locally finite. The category $\calO_c(W,\h)$ is a
highest weight category. In particular, it is artinian. Write
$\Irr(W)$ for the set of isomorphism classes of irreducible
representations of $W$. The poset of standard modules in
$\calO_c(W,\h)$ is indexed by $\Irr(W)$ with the partial order given
by \cite[Theorem 2.19]{GGOR}. More precisely, for $\xi\in\Irr(W)$,
equip it with a $\C W\ltimes\C[\h^\ast]$-module structure by letting
the elements in $\h\subset\C[\h^\ast]$ act by zero, the standard
module corresponding to $\xi$ is
$$\Delta(\xi)=H_c(W,\h)\otimes_{\C W\ltimes\C[\h^\ast]}\xi.$$
It is an indecomposable module with a simple head $L(\xi)$. The set
of isomorphism classes of simple modules in $\calO_c(W,\h)$ is
$$\{[L(\xi)]:\xi\in\Irr(W)\}.$$
It is a basis of the $\C$-vector space $K(\calO_c(W,\h))$. The set
$\{[\Delta(\xi)]:\xi\in\Irr(W)\}$ gives another basis of
$K(\calO_c(W,\h))$.

We say a module $N$ in $\calO_c(W,\h)$ has a standard filtration if
it admits a filtration
$$0=N_0\subset N_1\subset\ldots\subset N_n=N$$ such that each
quotient $N_i/N_{i-1}$ is isomorphic to a standard module. We denote
by $\calO^\Delta_c(W,\h)$ the full subcategory of $\calO_c(W,\h)$
consisting of such modules.

\begin{lemme}\label{standfilt}
(1)Any projective object in $\calO_c(W,\h)$ has a standard
filtration.

(2)A module in $\calO_c(W,\h)$ has a standard filtration if and only
if it is free as a $\C[\h]$-module.
\end{lemme}
Both (1) and (2) are given by \cite[Proposition 2.21]{GGOR}.

The category $\calO_c(W,\h)$ has enough projective objects and has
finite homological dimension \cite[Section 4.3.1]{GGOR}. In
particular, any module in $\calO_c(W,\h)$ has a finite projective
resolution. Write $\Proj_c(W,\h)$ for the full subcategory of
projective modules in $\calO_c(W,\h)$. Let
\begin{equation*}
I: \Proj_c(W,\h)\ra \calO_c(W,\h)
\end{equation*}
be the canonical embedding functor. We have the following lemma.

\begin{lemme}\label{projiso}
For any abelian category $\calA$ and any right exact
functors $F_1$, $F_2$ from $\calO_c(W,\h)$ to $\mathcal{A}$, the
homomorphism of vector spaces
\begin{equation*}
r_I:\Hom(F_1, F_2)\ra \Hom(F_1\circ I, F_2\circ I),
\quad\gamma\mapsto \gamma1_I
\end{equation*}
is an isomorphism.
\end{lemme}
In particular, if the functor $F_1\circ I$ is isomorphic to
$F_2\circ I$, then we have $F_1\cong F_2$.
\begin{proof}
We need to show that for any morphism of functors $\nu: F_1\circ
I\ra F_2\circ I$ there is a unique morphism $\tilde{\nu}: F_1\ra
F_2$ such that $\tilde{\nu}1_{I}=\nu$. Since $\calO_c(W,\h)$ has enough projectives, for any $M\in \calO_c(W,\h)$ there exists $P_0$, $P_1$ in $\Proj_c(W,\h)$ and an exact sequence in $\calO_c(W,\h)$
\begin{equation}\label{eq:projresolution}
P_1\overset{d_1}\lra P_0\overset{d_0}\lra M\lra 0.
\end{equation}
Applying the right exact functors $F_1$, $F_2$ to this sequence we get the two exact sequences in the diagram below. The morphism of functors $\nu:F_1\circ I\ra F_2\circ I$ yields well defined morphisms $\nu(P_1)$, $\nu(P_0)$ such that the square commutes
$$\xymatrix{F_1(P_1)\ar[r]^{F_1(d_1)}\ar[d]^{\nu(P_1)} & F_1(P_0)\ar[r]^{F_1(d_0)}\ar[d]^{\nu(P_0)} &F_1(M)\ar[r] \ar@{}[d] &0\ar@{}[d]\\F_2(P_1)\ar[r]^{F_2(d_1)} & F_2(P_0)\ar[r]^{F_2(d_0)} &F_2(M)\ar[r] &0.}$$
Define $\tilde{\nu}(M)$ to be the unique morphism $F_1(M)\ra F_2(M)$ that makes the diagram commute. Its definition is independent of the choice of $P_0$, $P_1$, and it is independent of the choice of the exact sequence (\ref{eq:projresolution}). The assignment $M\mapsto \tilde{\nu}(M)$ gives a morphism of functor $\tilde{\nu}: F_1\ra F_2$ such that $\tilde{\nu}1_{I}=\nu$. It is unique by the uniqueness of the morphism $\tilde{\nu}(M)$.
\end{proof}

\subsection{KZ functor.}\label{ss:KZ}

The Knizhnik-Zamolodchikov functor is an exact functor from the
category $\calO_c(W,\h)$ to the category $\Hecke_q(W,\h)\modu$,
where $q$ is a certain parameter associated with $c$. Let us recall
its definition from \cite[Section 5.3]{GGOR}.

Let $\mathcal{D}(\h_{reg})$ be the algebra of differential operators
on $\h_{reg}$. Write
$$H_c(W,\h_{reg})=H_c(W,\h)\otimes_{\C[\h]}\C[\h_{reg}].$$ We consider the Dunkl isomorphism, which is an
isomorphism of algebras
\begin{equation*}
H_c(W,\h_{reg})\overset{\sim}\ra \mathcal{D}(\h_{reg})\rtimes\C W
\end{equation*}
given by $x\mapsto x$, $w\mapsto w$ for $x\in\h^\ast$, $w\in W$, and
\begin{equation*}
y\mapsto
\partial_y+\sum_{s\in\mathcal{S}}\frac{2c_s}{1-\lam_s}\frac{\al_s(y)}{\al_s}(s-1),\quad\text{for }y\in\h.
\end{equation*}

For any $M\in \calO_c(W,\h)$, write
$$M_{\h_{reg}}=M\otimes_{\C[\h]}\C[\h_{reg}].$$
It identifies via the Dunkl isomorphism with a
$\mathcal{D}(\h_{reg})\rtimes W$-module which is finitely generated
over $\C[\h_{reg}]$. Hence $M_{\h_{reg}}$ is a $W$-equivariant
vector bundle on $\h_{reg}$ with an integrable connection $\nabla$
given by $\nabla_y(m)=\partial_ym$ for $m\in M$, $y\in\h$. It is
proved in \cite[Proposition 5.7]{GGOR} that the connection $\nabla$
has regular singularities. Now, regard $\h_{reg}$ as a complex
manifold endowed with the transcendental topology. Denote by
$\mathcal{O}^{an}_{\h_{reg}}$ the sheaf of holomorphic functions on
$\h_{reg}$. For any free $\C[\h_{reg}]$-module $N$ of finite rank,
we consider
$$N^{an}=N\otimes_{\C[\h_{reg}]}\mathcal{O}^{an}_{\h_{reg}}.$$ It is an
analytic locally free sheaf on $\h_{reg}$. For $\nabla$ an
integrable connection on $N$, the sheaf of holomorphic horizontal
sections
\begin{equation*}
N^{\nabla}=\{n\in N^{an}:\,\nabla_y(n)=0\text{ for all }y\in\h\}
\end{equation*}
is a $W$-equivariant local system on $\h_{reg}$. Hence it identifies
with a local system on $\h_{reg}/W$. So it yields a finite
dimensional representation of $\C B(W,\h)$. For $M\in \calO_c(W,\h)$
it is proved in \cite[Theorem 5.13]{GGOR} that the action of $\C
B(W,\h)$ on $(M_{\h_{reg}})^{\nabla}$ factors through the Hecke
algebra $\Hecke_q(W,\h)$. The formula for the parameter $q$ is given
in \cite[Section 5.2]{GGOR}.

The Knizhnik-Zamolodchikov functor is the functor
$$\KZ(W,\h): \calO_c(W,\h)\ra\Hecke_q(W,\h)\modu,\quad M\mapsto
(M_{\h_{reg}})^{\nabla}.$$ By definition it is exact. Let us recall
some of its properties following \cite{GGOR}. Assume in the rest of
this subsection that \emph{the algebras $\Hecke_q(W,\h)$ and $\C W$
have the same dimension over $\C$}. We abbreviate $\KZ=\KZ(W,\h)$. The functor $\KZ$ is represented by a projective object $P_{\KZ}$ in $\calO_c(W,\h)$. More precisely, there is an algebra homomorphism
\begin{equation*}
  \rho:\Hecke_q(W,\h)\ra\End_{\calO_c(W,\h)}(P_{\KZ})^{\op}
\end{equation*}
such that $\KZ$ is isomorphic to the functor $\Hom_{\calO_c(W,\h)}(P_{\KZ},-)$. By \cite[Theorem 5.15]{GGOR} the homomorphism $\rho$ is an isomorphism. In particular $\KZ(P_{\KZ})$ is isomorphic to $\Hecke_q(W,\h)$ as $\Hecke_q(W,\h)$-modules.

Now, recall that the center of a category $\calC$ is the
algebra $Z(\calC)$ of endomorphisms of the identity functor
$Id_{\calC}$. So there is a canonical map
$$Z(\calO_c(W,\h))\ra\End_{\calO_c(W,\h)}(P_{\KZ}).$$
The composition of this map with $\rho^{-1}$ yields an algebra
homomorphism
\begin{equation*}
\gamma: Z(\calO_c(W,\h))\ra Z(\Hecke_q(W,\h)),
\end{equation*}
where $Z(\Hecke_q(W,\h))$ denotes the center of $\Hecke_q(W,\h)$.

\begin{lemme}\label{lem:center}
(1) The homomorphism $\gamma$ is an isomorphism.

(2) For a module $M$ in $\calO_c(W,\h)$ and an element $f$ in
$Z(\calO_c(W,\h))$ the morphism
$$\KZ(f(M)): \KZ(M)\ra\KZ(M)$$ is the multiplication by $\gamma(f)$.
\end{lemme}
See \cite[Corollary 5.18]{GGOR} for (1). Part (2) follows from the
construction of $\gamma$.

The functor $\KZ$ is a quotient functor, see \cite[Theorem 5.14]{GGOR}. Therefore it has a right adjoint $S:\Hecke_q(W,\h)\ra\calO_c(W,\h)$ such that the canonical adjunction map $\KZ\circ S\ra\Id_{\Hecke_q(W,\h)}$ is an isomorphism of functors. We have the following proposition.

\begin{prop}\label{KZ}
Let $Q$ be a projective object in $\calO_c(W,\h)$.

(1) For any object $M\in\calO_c(W,\h)$, the following morphism of $\C$-vector spaces is an isomorphism
$$\Hom_{\calO_c(W,\h)}(M,Q)\simlra \Hom_{\Hecke_q(W)}(\KZ(M),\KZ(Q)),\quad f\mapsto\KZ(f).$$
In particular, the functor $\KZ$ is fully faithful over $\Proj_c(W,\h)$.

(2)The canonical adjunction map gives an isomorphism $Q\simra S\circ \KZ (Q)$.
\end{prop}
See \cite[Theorems 5.3, 5.16]{GGOR}.

\subsection{Parabolic restriction and induction for rational
DAHA's.}\label{ss:resDAHA}

From now on we will always assume that $\h^W=1$. Recall from Section
\ref{ss:resHecke} that $W'\subset W$ is the stabilizer of a point
$b\in\h$ and that $\overline{\h}=\h/\h^{W'}$. Let us recall from
\cite{BE} the definition of the parabolic restriction and induction
functors
$$\Res_b:\calO_c(W,\h)\ra\calO_{c'}(W',\overline{\h})\,,\quad
\Ind_b:\calO_{c'}(W',\overline{\h})\ra\calO_c(W,\h).$$ First we need some notation. For any point $p\in\h$ we write
$\C[[\h]]_p$ for the completion of $\C[\h]$ at $p$, and we write $\widehat{\C[\h]}_p$ for the completion of $\C[\h]$ at the $W$-orbit of $p$ in $\h$. Note that we have $\C[[\h]]_0=\widehat{\C[\h]}_0$. For any
$\C[\h]$-module $M$ let $$\widehat{M}_p=\widehat{\C[\h]}_p\otimes_{\C[\h]}M.$$ The completions $\widehat{H}_{c}(W,\h)_b$, $\widehat{H}_{c'}(W',\h)_0$ are well defined algebras.
We denote by $\widehat{\calO}_c(W,\h)_b$ the
category of $\widehat{H}_{c}(W,\h)_b$-modules that are finitely
generated over $\widehat{\C[\h]}_b$, and we denote by $\widehat{\calO}_{c'}(W',\h)_0$ the
category of $\widehat{H}_{c'}(W',\h)_0$-modules that are finitely
generated over $\widehat{\C[\h]}_0$. Let
$P=\mathrm{Fun}_{W'}(W,\widehat{H}_{c}(W',\h)_0)$ be the set of
$W'$-invariant maps from $W$ to $\widehat{H}_{c}(W',\h)_0$. Let
$Z(W,W',\widehat{H}_{c}(W',\h)_0)$ be the ring of endomorphisms of the
right $\widehat{H}_{c}(W',\h)_0$-module $P$.
We have the following proposition given by \cite[Theorem 3.2]{BE}.
\begin{prop}\label{BEiso}
There is an isomorphism of algebras $$\Theta:
\widehat{H}_{c}(W,\h)_b\lra Z(W,W', \widehat{H}_{c'}(W',\h)_0)$$
defined as follows: for $f\in P$, $\al\in\h^\ast$, $a\in\h$, $u\in
W$,
\begin{eqnarray*}
(\Theta(u)f)(w)&=&f(w u),\\
(\Theta(x_{\al})f)(w)&=&(x^{(b)}_{w\al}+\al(w^{-1}b))f(w),\\
(\Theta(y_a)f)(w)&=&y^{(b)}_{wa}f(w)+\sum_{s\in\mathcal{S}, s\notin
W'}\frac{2c_s}{1-\lam_s}\frac{\al_s(wa)}{x^{(b)}_{\al_s}+\al_s
(b)}(f(sw)-f(w)),
\end{eqnarray*}
where $x_\al\in\h^\ast\subset H_{c}(W,\h)$,
$x^{(b)}_{\al}\in\h^\ast\subset H_{c'}(W',\h)$, $y_a\in\h\subset
H_{c}(W,\h)$, $y_a^{(b)}\in\h\subset H_{c'}(W',\h)$.
\end{prop}

Using $\Theta$ we will identify $\widehat{H}_{c}(W,\h)_b$-modules with $Z(W,W', \widehat{H}_{c'}(W',\h)_0)$-modules. So the module $P=\mathrm{Fun}_{W'}(W,\widehat{H}_{c}(W',\h)_0)$ becomes an
$(\widehat{H}_{c}(W,\h)_b,\widehat{H}_{c'}(W',\h)_0)$-bimodule. Hence
for any $N\in \widehat{\calO}_{c'}(W',\h)_0$ the module
$P\otimes_{\widehat{H}_{c'}(W',\h)_0}N$ lives in
$\widehat{\calO}_c(W,\h)_b$. It is naturally identified with
$\mathrm{Fun}_{W'}(W,N)$, the set of $W'$-invariant maps from $W$ to
$N$. For any $\C[\h^\ast]$-module $M$ write $E(M)\subset M$ for the
locally nilpotent part of $M$ under the action of $\h$.

The ingredients for defining the functors $\Res_b$ and $\Ind_b$
consist of:
\begin{itemize}
\item the adjoint pair of functors $(\widehat{\quad}_b, E^b)$ with
$$\widehat{\quad}_b:\calO_{c}(W,\h)\ra\widehat{\calO}_c(W,\h)_b,\quad
M\mapsto\widehat{M}_b,$$
$$E^b: \widehat{\calO}_c(W,\h)_b\ra \calO_{c}(W,\h),\quad N\ra
E(N),$$
\item the Morita equivalence
$$J:\widehat{\calO}_{c'}(W',\h)_0\ra\widehat{\calO}_c(W,\h)_b,\quad
N\mapsto \mathrm{Fun}_{W'}(W,N),$$ and its quasi-inverse
$R$ given in Section \ref{ss:BEiso} below,
\item the equivalence of categories
$$E: \widehat{\calO}_{c'}(W',\h)_0\ra\calO_{c'}(W',\h), \quad M\mapsto E(M)$$
and its quasi-inverse given by $N\mapsto\widehat{N}_0$,
\item the equivalence of categories
\begin{equation}\label{zeta}
\zeta: \calO_{c'}(W',\h)\ra \calO_{c'}(W',\overline{\h}),\quad
M\mapsto\{v\in M:\,yv=0,\,\text{ for all }y\in \h^{W'}\}
\end{equation}
and its quasi-inverse $\zeta^{-1}$ given in Section
\ref{rmq:resDAHA} below.
\end{itemize}
For $M\in\calO_c(W,\h)$ and $N\in\calO_{c'}(W',\overline{\h})$ the
functors $\Res_b$ and $\Ind_b$ are defined by
\begin{eqnarray}
\Res_b(M)=\zeta\circ E\circ R(\widehat{M}_b),\label{Resb}\\
\Ind_b(N)=E^b\circ J(\widehat{\zeta^{-1}(N)}_0).\nonumber
\end{eqnarray}
We refer to \cite[Section 2,3]{BE} for details.

\subsection{The idempotent $x_{\pr}$ and the functor $R$.}\label{ss:BEiso}

We give some details on the isomorphism $\Theta$ for a future use. Fix elements
$1=u_1, u_2,\ldots, u_r$ in $W$ such that $W=\bigsqcup_{i=1}^r
W'u_i$. Let $\mathrm{Mat}_r(\widehat{H}_{c'}(W',\h)_0)$ be the
algebra of $r\times r$ matrices with coefficients in
$\widehat{H}_{c'}(W',\h)_0$. We have an algebra isomorphism
\begin{eqnarray}
\Phi:Z(W,W', \widehat{H}_{c'}(W',\h)_0)&\ra&
\mathrm{Mat}_r(\widehat{H}_{c'}(W',\h)_0),\label{Phi}\\
A&\mapsto& (\Phi(A)_{ij})_{1\leqs i,j\leqs r}\nonumber
\end{eqnarray}
such that
\begin{equation*}
(Af)(u_i)=\sum_{j=1}^r\Phi(A)_{ij}f(u_j), \quad \text{ for all }f\in
P, \,1\leqs i\leqs r.\end{equation*}
Denote by $E_{ij}$, $1\leqs i,j\leqs r$, the elementary matrix in $\mathrm{Mat}_r(\widehat{H}_{c'}(W',\h)_0)$ with coefficient $1$ in the position $(i,j)$ and zero elsewhere. Note that the algebra isomorphism
$$\Phi\circ\Theta: \widehat{H}_{c}(W,\h)_b\simlra \mathrm{Mat}_r(\widehat{H}_{c'}(W',\h)_0)$$
restricts to an isomorphism of subalgebras
\begin{equation}\label{eq:phithetax}
 \widehat{\C[\h]}_b\cong\bigoplus_{i=1}^r \C[[\h]]_0E_{ii}.
\end{equation}
Indeed, there is an unique isomorphism of algebras
\begin{equation}\label{eq:varpi}
\varpi:\widehat{\C[\h]}_b\cong\bigoplus_{i=1}^r\C[[\h]]_{u_i^{-1}b}.
\end{equation}
extending the algebra homomorphism
$$\C[\h]\ra\bigoplus_{i=1}^r\C[\h],\quad x\mapsto (x,x,\ldots, x),\quad \forall\ x\in\h^\ast.$$
For each $i$ consider the isomorphism of algebras
$$\phi_i: \C[[\h]]_{u_i^{-1}b}\ra\C[[\h]]_0,\quad x\mapsto u_ix+x(u_i^{-1}b),\quad\forall\ x\in\h^\ast.$$
The isomorphism (\ref{eq:phithetax}) is exactly the composition of $\varpi$ with the direct sum $\oplus_{i=1}^r\phi_i.$ Here $E_{ii}$ is the image of the idempotent in $\widehat{\C[\h]}_b$ corresponding to the component $\C[[\h]]_{u_i^{-1}b}$. We will denote by $x_{\pr}$ the idempotent in $\widehat{\C[\h]}_b$ corresponding to $\C[[\h]]_b$, i.e., $\Phi\circ\Theta(x_{\pr})=E_{11}$. Then the following functor
$$R:\widehat{\calO}_c(W,\h)_b\ra \widehat{\calO}_{c'}(W',\h)_0,\quad M\mapsto x_{\pr}M$$
is a quasi-inverse of $J$. Here, the action of $\widehat{H}_{c'}(W',\h)_0$ on $R(M)=x_{\pr}M$ is given by the following formulas deduced from Proposition \ref{BEiso}. For any $\al\in\h^\ast$, $w\in W'$,
$a\in\h^\ast$, $m\in M$ we have
\begin{eqnarray}
 x_\al^{(b)}x_{\pr}(m)&=&x_{\pr}((x_{\al}-\al(b))m), \label{xform}\\
wx_{\pr}(m)&=&x_{\pr}(wm), \label{wform}\\
y_a^{(b)}x_{\pr}(m)&=&x_{\pr}((y_a+\sum_{s\in\mathcal{S},\,s\notin
W'}\frac{2c_s}{1-\lam_s}\frac{\al_s(a)}{x_{\al_s}})m). \label{yform}
\end{eqnarray}
In particular, we have
\begin{equation}\label{eq:R(M)}
  R(M)=\phi_1^\ast(x_{\pr}(M))
\end{equation}
as $\C[[\h]]_0\rtimes W'$-modules. Finally, note that the following equality holds in $\widehat{H}_c(W,\h)_b$
\begin{equation}\label{killwform}
x_{\pr}ux_{\pr}=0, \quad \forall\ u\in W-W'.
\end{equation}

\subsection{A quasi-inverse of $\zeta$.}\label{rmq:resDAHA}

Let us recall from \cite[Section 2.3]{BE} the following facts. Let
$\h^{\ast W'}$ be the subspace of $\h^\ast$ consisting of fixed
points of $W'$. Set $$(\h^{\ast W'})^\bot=\{v\in\h: f(v)=0\text{ for
all } f\in\h^{\ast W'}\}.$$ We have a $W'$-invariant decomposition
$$\h=(\h^{\ast W'})^\bot\oplus\h^{W'}.$$ The $W'$-space $(\h^{\ast
W'})^\bot$ is canonically identified with $\overline{\h}$. Since the
action of $W'$ on $\h^{W'}$ is trivial, we have an obvious algebra
isomorphism
\begin{equation}\label{isobete}
H_{c'}(W',\h)\cong H_{c'}(W',\overline{\h})\otimes
\mathcal{D}(\h^{W'}).
\end{equation}
It maps an element $y$ in the subset $\h^{W'}$ of $H_{c'}(W',\h)$ to
the operator $\partial_y$ in $\calD(\h^{W'})$. Write
$\calO(1,\h^{W'})$ for the category of finitely generated
$\calD(\h^{W'})$-modules that are $\partial_y$-locally nilpotent for
all $y\in\h^{W'}$. The algebra isomorphism above yields an
equivalence of categories
\begin{equation*}
\calO_{c'}(W',\h)\cong\calO_{c'}(W',\overline{\h})\otimes\calO(1,\h^{W'}).
\end{equation*}
The functor $\zeta$ in (\ref{zeta}) is an equivalence, because it is
induced by the functor
$$\calO(1,\h^{W'})\simra\C\modu,\quad M\ra\{m\in M, \partial_y(m)=0\text{ for all }y\in\h^{W'}\},$$
which is an equivalence by Kashiwara's lemma upon taking Fourier
transforms. In particular, a quasi-inverse of $\zeta$ is given by
\begin{equation}\label{eq:zetainverse}
\zeta^{-1}: \calO_{c'}(W',\overline{\h})\ra \calO_{c'}(W',\h),\quad N\mapsto
N\otimes\C[\h^{W'}],\end{equation}
where $\C[\h^{W'}]\in \calO(1,\h^{W'})$ is
the polynomial representation of $\mathcal{D}(\h^{W'})$.

Moreover, the functor $\zeta$ maps a standard module in
$\calO_{c'}(W',\h)$ to a standard module in
$\calO_{c'}(W',\overline{\h})$. Indeed, for any $\xi\in\Irr(W')$, we
have an isomorphism of $H_{c'}(W',\h)$-modules
\begin{equation*}
H_{c'}(W',\h)\otimes_{\C[\h^\ast]\rtimes
W'}\xi=(H_{c'}(W',\overline{\h})\otimes_{\C[(\overline{\h})^\ast]\rtimes
W'}\xi)\otimes(\calD(\h^{W'})\otimes_{\C[(\h^{W'})^\ast]}\C).
\end{equation*}
On the right hand side $\C$ denotes the trivial module of
$\C[(\h^{W'})^\ast]$, and the latter is identified with the
subalgebra of $\calD(\h^{W'})$ generated by $\partial_y$ for all
$y\in\h^{W'}$. We have
$$\calD(\h^{W'})\otimes_{\C[(\h^{W'})^\ast]}\C\cong\C[\h^{W'}]$$ as
$\calD(\h^{W'})$-modules. So $\zeta$ maps the standard module
$\Delta(\xi)$ for $H_{c'}(W',\h)$ to the standard module
$\Delta(\xi)$ for $H_{c'}(W',\overline{\h})$.

\subsection{}\label{ss:resprop}

Here are some properties of $\Res_b$ and $\Ind_b$.
\begin{prop}\label{Res}
\begin{itemize}
\item[(1)] Both functors $\Res_b$ and $\Ind_b$ are exact. The functor $\Res_b$ is left
adjoint to $\Ind_b$. In particular the functor $\Res_b$ preserves
projective objects and $\Ind_b$ preserves injective objects.
\item[(2)] Let $\Res^W_{W'}$
and $\Ind^W_{W'}$ be respectively the restriction and induction
functors of groups. We have the following commutative diagram
\begin{equation*}
\xymatrix{K(\calO_c(W,\h))\ar[r]_{\sim}^{\omega}\ar@<1ex>[d]^{\Res_b} & K(\C W)\ar@<1ex>[d]^{\Res^W_{W'}}\\
      K(\calO_{c'}(W',\overline{\h}))\ar[r]_{\sim}^{\omega'}\ar@<1ex>[u]^{\Ind_b} & K(\C
W')\ar@<1ex>[u]^{\Ind^W_{W'}}.}
\end{equation*}
Here the isomorphism $\omega$ (resp. $\omega'$) is given by mapping
$[\Delta(\xi)]$ to $[\xi]$ for any $\xi\in\Irr(W)$ (resp.
$\xi\in\Irr(W')$).
\end{itemize}
\end{prop}

See \cite[Proposition $3.9$, Theorem $3.10$]{BE} for (1),
\cite[Proposition $3.14$]{BE} for (2).

\subsection{Restriction of modules having a standard filtration}\label{ss:standardres}

In the rest of Section 1, we study the actions of the restriction functors on modules having a standard filtration in
$\calO_c(W,\h)$ (Proposition \ref{standard}). We will need the following lemmas.

\begin{lemme}\label{lem:MV}
Let $M$ be a module in $\calO^\Delta_c(W,\h)$.

(1) There is a finite dimensional subspace $V$ of $M$ such that $V$
is stable under the action of $\C W$ and the map
\begin{equation*}
\C[\h]\otimes V\ra M,\quad p\otimes v\mapsto pv
\end{equation*}
is an isomorphism of $\C[\h]\rtimes W$-modules.

(2) The map $\omega:K(\calO_c(W,\h))\ra K(\C W)$ in Proposition
\ref{Res}(2) satisfies
\begin{equation}\label{eq:omegakgp}
\omega([M])=[V].
\end{equation}
\end{lemme}
\begin{proof}
Let $$0=M_0\subset M_1\subset \ldots\subset M_l=M$$ be a filtration
of $M$ such that for any $1\leqs i\leqs l$ we have
$M_i/M_{i-1}\cong\Delta(\xi_i)$ for some $\xi_i\in\Irr(W)$. We
prove (1) and (2) by recurrence on $l$. If $l=1$, then $M$ is a
standard module. Both (1) and (2) hold by definition. For $l>1$, by
induction we may suppose that there is a subspace $V'$ of $M_{l-1}$
such that the properties in (1) and (2) are satisfied for $M_{l-1}$
and $V'$. Now, consider the exact sequence
$$0\lra M_{l-1}\lra M\overset{j}\lra \Delta(\xi_l)\lra 0$$
From the isomorphism of $\C[\h]\rtimes W$-modules
$\Delta(\xi_l)\cong\C[\h]\otimes \xi$ we see that $\Delta(\xi_l)$
is a projective $\C[\h]\rtimes W$-module. Hence there exists a
morphism of $\C[\h]\rtimes W$-modules $s: \Delta(\xi_l)\ra M$ that
provides a section of $j$. Let $V=V'\oplus s(\xi_l)\subset M$. It is
stable under the action of $\C W$. The map $\C[\h]\otimes V\ra M$ in
(1) is an injective morphism of $\C[\h]\rtimes W$-modules. Its image
is $M_{l-1}\oplus s(\Delta(\xi))$, which is equal to $M$. So it is
an isomorphism. We have
$$\omega([M])=\omega([M_{l-1}])+\omega([\Delta(\xi_l)]),$$ by
assumption $\omega([M_{l-1}])=[V']$, so
$\omega([M])=[V']+[\xi_l]=[V]$.
\end{proof}

\begin{lemme}\label{lem:eufinite}
(1) Let $M$ be a
$\widehat{H}_{c}(W,\h)_0$-module free over $\C[[\h]]_0$. If there
exist generalized eigenvectors $v_1,\ldots v_n$ of $\eu$ which form
a basis of $M$ over $\C[[\h]]_0$, then for $f_1,\ldots,
f_n\in\C[[\h]]_0$ the element $m=\sum_{i=1}^nf_iv_i$ is $\eu$-finite
if and only if $f_1,\ldots,f_n$ all belong to $\C[\h]$.

(2) Let $N$ be an object in $\calO_c(W,\h)$. If $\widehat{N}_0$ is a
free $\C[[\h]]_0$-module, then $N$ is a free $\C[\h]$-module. It
admits a basis consisting of generalized eigenvectors
$v_1,\ldots,v_n$ of $\eu$.
\end{lemme}
\begin{proof}
(1) It follows from the proof of \cite[Theorem 2.3]{BE}.

(2) Since $N$ belongs to $\calO_c(W,\h)$, it is finitely generated
over $\C[\h]$. Denote by $\mathfrak{m}$ the maximal ideal of $\C[[\h]]_0$.
The canonical map $N\ra\widehat{N}_0/\mathfrak{m}\widehat{N}_0$ is
surjective. So there exist $v_1,\ldots,v_n$
in $N$ such that their images form a basis of
$\widehat{N}_0/\mathfrak{m}\widehat{N}_0$ over $\C$. Moreover, we may choose $v_1,\ldots,v_n$ to be generalized eigenvectors of $\eu$, because the $\eu$-action on $N$ is locally finite. Since
$\widehat{N}_0$ is free over $\C[[\h]]_0$, Nakayama's lemma yields
that $v_1,\ldots,v_n$ form a basis of $\widehat{N}_0$ over
$\C[[\h]]_0$. By part (1) the set $N'$ of $\eu$-finite elements in
$\widehat{N}_0$ is the free $\C[\h]$-submodule generated by
$v_1,\ldots, v_n$. On the other hand, since $\widehat{N}_0$ belongs
to $\widehat{\calO}_c(W,\h)_0$, by \cite[Proposition 2.4]{BE} an
element in $\widehat{N}_0$ is $\h$-nilpotent if and only if it is
$\eu$-finite. So $N'=E(\widehat{N}_0).$ On the other hand, the canonical inclusion $N\subset E(\widehat{N}_0)$ is an equality by \cite[Theorem 3.2]{BE}. Hence $N=N'$. This implies that $N$ is free over
$\C[\h]$, with a basis given by $v_1,\ldots,v_n$, which are
generalized eigenvectors of $\eu$.
\end{proof}

\begin{prop}\label{standard}
Let $M$ be an object in $\calO^\Delta_c(W,\h)$.

(1) The object $\Res_b(M)$ has a standard filtration.

(2) Let $V$ be a subspace of $M$ that has the properties of Lemma
\ref{lem:MV}(1). Then there is an isomorphism of
$\C[\overline{\h}]\rtimes W'$-modules
\begin{equation*}
\Res_b(M)\cong \C[\overline{\h}]\otimes\Res^{W}_{W'}(V).
\end{equation*}
\end{prop}
\begin{proof}
(1) By the end of Section \ref{rmq:resDAHA} the equivalence
$\zeta$ maps a standard module in $\calO_{c'}(W',\h)$ to a standard
one in $\calO_{c'}(W',\overline{\h})$. Hence to prove that
$\Res_b(M)=\zeta\circ E\circ R(\widehat{M}_b)$ has a standard
filtration, it is enough to show that $N=E\circ R(\widehat{M}_b)$
has one. We claim that the module $N$ is free over $\C[\h]$. So the
result follows from Lemma \ref{standfilt}(2).

Let us prove the claim. Recall from (\ref{eq:R(M)}) that we have $R(\widehat{M}_b)=\phi_1^\ast(x_{\pr}\widehat{M}_b)$ as $\C[[\h]]_0\rtimes W'$-modules. Using the isomorphism of $\C[\h]\rtimes W$-modules
$M\cong\C[\h]\otimes V$ given in Lemma \ref{lem:MV}(1), we deduce an isomorphism of $\C[[\h]]_0\rtimes W'$-modules
\begin{eqnarray*}
R(\widehat{M}_b)&\cong&\phi_1^\ast(x_{\pr}(\widehat{\C[\h]}_b\otimes V))\nonumber\\
&\cong & \C[[\h]]_0\otimes V.\label{completeiso}
\end{eqnarray*}
So the module $R(\widehat{M}_b)$ is free over $\C[[\h]]_0$. The completion
of the module $N$ at $0$ is isomorphic to $R(\widehat{M}_b)$. By
Lemma \ref{lem:eufinite}(2) the module $N$ is free over $\C[\h]$.
The claim is proved.

(2) Since $\Res_b(M)$ has a standard filtration, by Lemma
\ref{lem:MV} there exists a finite dimensional vector space
$V'\subset \Res_b(M)$ such that $V'$ is stable under the action of
$\C W'$ and we have an isomorphism of $\C[\overline{\h}]\rtimes
W'$-modules
\begin{equation*}
\Res_b(M)\cong \C[\overline{\h}]\otimes V'.
\end{equation*}
Moreover, we have $\omega'([\Res_b(M)])=[V']$ where $\omega'$ is the
map in Proposition \ref{Res}(2). The same proposition yields that
$\Res^W_{W'}(\omega[M])=\omega'([\Res_b(M)])$. Since
$\omega([M])=[V]$ by (\ref{eq:omegakgp}), the $\C W'$-module $V'$ is
isomorphic to $\Res^W_{W'}(V)$. So we have an isomorphism of
$\C[\overline{\h}]\rtimes W'$-modules
$$\Res_b(M)\cong \C[\overline{\h}]\otimes\Res^{W}_{W'}(V).$$
\end{proof}

\section{KZ commutes with restriction functors}\label{s:KZcommute}

In this section, we relate the restriction and induction functors
for rational DAHA's to the corresponding functors for Hecke algebras
via the functor $\KZ$. We will always assume that the Hecke algebras
have the same dimension as the corresponding group algebras. Thus
the Knizhnik-Zamolodchikov functors admit the properties recalled in
Section \ref{ss:KZ}.

\subsection{}\label{ss: thmiso}

Let $W$ be a complex reflection group acting on $\h$. Let $b$ be a
point in $\h$ and let $W'$ be its stabilizer in $W$. We will
abbreviate $\KZ=\KZ(W,\h)$, $\KZ'=\KZ(W',\overline{\h})$.
\begin{thm}\label{iso}
There is an isomorphism of functors
\begin{equation*}
 \KZ'\circ\Res_b\cong\Resh\circ\KZ.
\end{equation*}
\end{thm}
\begin{proof}
We will regard $\KZ: \calO_c(W,\h)\ra \Hecke_q(W)\modu$ as a functor
from $\calO_c(W,\h)$ to $B_W\modu$ in the obvious way. Similarly we
will regard $\KZ'$ as a functor to $B_{W'}\modu$. Recall the
inclusion $\imath: B_{W'}\hookrightarrow B_W$ from
(\ref{heckeres4}). The theorem amounts to prove that for any $M\in
\calO_c(W,\h)$ there is a natural isomorphism of $B_{W'}$-modules
\begin{equation}\label{eq:thm}
\KZ'\circ\Res_b(M)\cong\imath_\ast\circ\KZ(M).\end{equation}

\emph{Step 1.} Recall the functor $\zeta:
\calO_{c'}(W',\h)\ra\calO_{c'}(W',\overline{\h})$ from (\ref{zeta}) and its quasi-inverse $\zeta^{-1}$ in (\ref{eq:zetainverse}).
Let $$N=\zeta^{-1}(\Res_b(M)).$$ We have
$N\cong\Res_b(M)\otimes\C[\h^{W'}].$ Since the canonical
epimorphism $\h\ra\overline{\h}$ induces a fibration
$\h'_{reg}\ra\overline{\h}_{reg}$, see Section \ref{ss:resHecke}, we have
\begin{equation}\label{h}
N_{\h'_{reg}}\cong
\Res_b(M)_{\overline{\h}_{reg}}\otimes\C[\h^{W'}].\end{equation} By
Dunkl isomorphisms, the left hand side is a $\mathcal{D}(\h'_{reg})\rtimes
W'$-module while the right hand side is a
$(\mathcal{D}(\overline{\h}_{reg})\rtimes
W')\otimes\mathcal{D}(\h^{W'})$-module. Identify these two algebras
in the obvious way. The isomorphism (\ref{h}) is compatible with the
$W'$-equivariant $\mathcal{D}$-module structures. Hence we have
$$(N_{\h'_{reg}})^\nabla\cong(\Res_b(M)_{\overline{\h}_{reg}})^\nabla\otimes
\C[\h^{W'}]^\nabla.$$ Since $\C[\h^{W'}]^\nabla=\C$, this yields a
natural isomorphism
\begin{equation*}
\ell_\ast\circ\KZ(W',\h)(N)\cong\KZ'\circ\Res_b(M),
\end{equation*}
where $\ell$ is the homomorphism defined in (\ref{heckeres1}).

\emph{Step 2.} Consider the $W'$-equivariant algebra isomorphism
$$\phi: \C[\h]\ra\C[\h],\quad x\mapsto x+x(b)\text{ for } x\in\h^\ast.$$
It induces an isomorphism
$\hat{\phi}:\C[[\h]]_b\overset{\sim}\ra\C[[\h]]_0$. The latter
yields an algebra isomorphism
$$\C[[\h]]_b\otimes_{\C[\h]}\C[\h_{reg}]\simeq
\C[[\h]]_0\otimes_{\C[\h]}\C[\h'_{reg}].$$ To see this note first
that by definition, the left hand side is $\C[[\h]]_b[\al_{s}^{-1},
s\in\mathcal{S}]$. For $s\in\mathcal{S}$, $s\notin W'$ the element
$\al_{s}$ is invertible in $\C[[\h]]_b$, so we have
$$\C[[\h]]_b\otimes_{\C[\h]}\C[\h_{reg}]=\C[[\h]]_b[\al_{s}^{-1},
s\in\mathcal{S}\cap W'].$$ For $s\in \mathcal{S}\cap W'$ we have
$\al_{s}(b)=0$, so $\hat{\phi}(\al_s)=\al_s$. Hence
\begin{eqnarray*}
\hat{\phi}(\C[[\h]]_b)[\hat{\phi}(\al_{s})^{-1},
s\in\mathcal{S}\cap W']&=&\C[[\h]]_0[\al_{s}^{-1},
s\in\mathcal{S}\cap
W']\\
&=&\C[[\h]]_0\otimes_{\C[\h]}\C[\h'_{reg}].
\end{eqnarray*}

\emph{Step 3.} We will assume in Steps 3, 4, 5 that $M$ is a module
in $\calO^\Delta_c(W,\h)$. In this step we prove that $N$ is
isomorphic to $\phi^\ast(M)$ as $\C[\h]\rtimes W'$-modules. Let
$V$ be a subspace of $M$ as in Lemma \ref{lem:MV}(1). So we have an
isomorphism of $\C[\h]\rtimes W$-modules
\begin{equation}\label{inducefilt}
M\cong \C[\h]\otimes V.
\end{equation}
Also, by Proposition \ref{standard}(2) there is an isomorphism of
$\C[\h]\rtimes W'$-modules
\begin{eqnarray*}
N&\cong&\C[\h]\otimes\Res^W_{W'}(V).
\end{eqnarray*}
So $N$ is isomorphic to $\phi^\ast(M)$ as $\C[\h]\rtimes
W'$-modules.

\emph{Step 4.} In this step we compare
$(\widehat{(\phi^\ast(M))}_0)_{\h'_{reg}}$ and
$(\widehat{N}_0)_{\h'_{reg}}$ as
$\widehat{\calD(\h'_{reg})}_0$-modules. The definition of these
$\widehat{\calD(\h'_{reg})}_0$-module structures will be given below
in terms of connections. By (\ref{Resb}) we have $N=E\circ
R(\widehat{M}_b)$, so we have $\widehat{N}_0\cong
R(\widehat{M}_b).$ Next, by (\ref{eq:R(M)}) we have an isomorphism of $\C[[\h]]_0\rtimes W'$-modules
\begin{eqnarray*}
  R(\widehat{M}_b)&=&\hat{\phi}^\ast (x_{\pr}(\widehat{M}_b))\\
  &=&\widehat{(\phi^\ast(M))}_0.
\end{eqnarray*}
So we get an isomorphism of $\C[[\h]]_0\rtimes W'$-modules
$$\hat\Psi: \widehat{(\phi^\ast(M))}_0\ra\widehat{N}_0.$$
Now, let us consider connections on these modules. Note that by Step $2$ we have
$$(\widehat{(\phi^\ast(M))}_0)_{\h'_{reg}}
=\hat{\phi}^\ast(x_{\pr}(\widehat{M}_b)_{\h_{reg}}).$$
Write $\nabla$ for the connection on $M_{\h_{reg}}$ given by the Dunkl isomorphism for $H_c(W,\h_{reg})$. We equip
$(\widehat{(\phi^\ast(M))}_0)_{\h'_{reg}}$ with the connection $\tilde{\nabla}$ given by
$$\tilde{\nabla}_a(x_{\pr}m)=x_{\pr}(\nabla_a(m)),\quad\forall\ m\in(\widehat{M}_b)_{\h_{reg}},\ a\in\h.$$ Let $\nabla^{(b)}$ be the connection on $N_{\h'_{reg}}$ given by the Dunkl isomorphism for $H_{c'}(W',\h'_{reg})$. This restricts to a connection on $(\widehat{N}_0)_{\h'_{reg}}$. We claim that $\Psi$ is compatible with these connections, i.e., we have
\begin{equation}\label{but}
\nabla_a^{(b)}(x_{\pr} m)=x_{\pr}\nabla_a(m),\quad\forall\ m\in (\widehat{M}_b)_{\h_{reg}}.
\end{equation} Recall the subspace $V$ of $M$ from Step 3. By
Lemma \ref{lem:MV}(1) the map
$$(\widehat{\C[\h]}_b\otimes_{\C[\h]}\C[\h_{reg}])\otimes V\ra
(\widehat{M_b})_{\h_{reg}},\quad p\otimes v\mapsto pv$$ is a
bijection. So it is enough to prove (\ref{but}) for $m=pv$ with
$p\in\widehat{\C[\h]}_b\otimes_{\C[\h]}\C[\h_{reg}]$, $v\in V$. We have
\begin{eqnarray}\label{conncal}
\nabla^{(b)}_a(x_{\pr}p v)&=&(y^{(b)}_a-\sum_{s\in\mathcal{S}\cap
W'}\frac{2c_s}{1-\lam_s}\frac{\al_s(a)}
{x_{\al_s}^{(b)}}(s-1))(x_{\pr}p v)\nonumber\\
&=&x_{\pr}(y_a+\sum_{s\in\mathcal{S},s\notin
W'}\frac{2c_s}{1-\lam_s}\frac{\al_s(a)}{x_{\al_s}}-\nonumber\\
&&-\sum_{s\in\mathcal{S}\cap W'}\frac{2c_s}{1-\lam_s}\frac{\al_s(a)}{x_{\al_s}}(s-1))(x_{\pr}p v)\nonumber\\
&=&x_{\pr}(\nabla_a+\sum_{s\in\mathcal{S},s\notin W'}\frac{2c_s}{1-\lam_s}\frac{\al_s(a)}{x_{\al_s}}s)(x_{\pr}p v)\nonumber\\
&=&x_{\pr}\nabla_a(x_{\pr}p v) .
\end{eqnarray}
Here the first equality is by the Dunkl isomorphism for
$H_{c'}(W',\h'_{reg})$. The second is by (\ref{xform}),
(\ref{wform}), (\ref{yform}) and the fact that $x_{\pr}^2=x_{\pr}$.
The third is by the Dunkl isomorphism for $H_{c}(W,\h_{reg})$. The last is by (\ref{killwform}). Next, since $x_{\pr}$ is the idempotent in $\widehat{\C[\h]}_b$ corresponding to the component $\C[[\h]]_b$ in the decomposition (\ref{eq:varpi}), we have
\begin{eqnarray*}
\nabla_a(x_{\pr}p
v)&=&(\partial_a(x_{\pr}p))v+x_{\pr}p\,(\nabla_av)\\
&=&x_{\pr}(\partial_a(p))v+x_{\pr}p\,(\nabla_av)\\
&=&x_{\pr}\nabla_a(p v).
\end{eqnarray*}
Together with (\ref{conncal}) this implies that
$$\nabla^{(b)}_a(x_{\pr}p v)=x_{\pr}\nabla_a(p v).$$
So (\ref{but}) is proved.

\emph{Step 5.} In this step we prove isomorphism (\ref{eq:thm}) for
$M\in\calO^\Delta_c(W,\h)$. Here we need some more notation. For
$X=\h$ or $\h'_{reg}$, let $U$ be an open analytic subvariety of $X$, write
$i:U\hookrightarrow X$ for the canonical embedding. For $F$ an analytic coherent sheaf
on $X$ we write $i^\ast (F)$ for the restriction of $F$ to $U$. If
$U$ contains $0$, for an analytic locally free sheaf $E$ over $U$,
we write $\widehat{E}$ for the restriction of $E$ to the formal disc
at $0$.

Let $\Omega\subset \h$ be the open ball defined in (\ref{eq:omega}).
Let $f:\h\ra\h$ be the morphism defined by $\phi$. It maps
$\Omega$ to an open ball $\Omega_0$ centered at $0$. We have
$$f(\Omega\cap\h_{reg})=\Omega_0\cap\h'_{reg}.$$
Let $u:\Omega_0\cap\h'_{reg}\hookrightarrow\h$ and $v:
\Omega\cap\h_{reg}\hookrightarrow\h$ be the canonical embeddings. By
Step 3 there is an isomorphism of $W'$-equivariant analytic locally
free sheaves over $\Omega_0\cap\h'_{reg}$
$$u^\ast (N^{an})\cong
\phi^\ast(v^\ast (M^{an})).$$ By Step 4 there is an isomorphism
$$\widehat{u^\ast (N^{an})}\simra\widehat{\phi^\ast(v^\ast
(M^{an}))}$$ which is compatible with their connections. It follows
from Lemma \ref{monodromie} below that there is an isomorphism
$$(u^\ast (N^{an}))^{\nabla^{(b)}}
\cong \phi^\ast((v^\ast (M^{an}))^{\nabla}).$$ Since
$\Omega_0\cap\h'_{reg}$ is homotopy equivalent to $\h'_{reg}$ via
$u$, the left hand side is isomorphic to $(N_{\h'_{reg}})^{\nabla^{(b)}}$. So
we have
\begin{equation*}
\kappa_\ast\circ\jmath_\ast\circ\KZ(M)\cong\KZ(W',\h)(N),
\end{equation*}
where $\kappa$, $\jmath$ are as in (\ref{heckeres2}), (\ref{heckeres3}). Combined with Step 1 we have the following isomorphisms
\begin{eqnarray}\label{i}
\KZ'\circ\Res_b(M)&\cong&\ell_\ast\circ\KZ(W',\h)(N)\nonumber\\
&\cong&\ell_\ast\circ\kappa_\ast\circ\jmath_\ast\circ\KZ(M)\\
&=&\imath_\ast\circ\KZ(M).\nonumber
\end{eqnarray}
They are functorial on $M$.

\begin{lemme}\label{monodromie}
Let $E$ be an analytic locally free sheaf over the complex manifold
$\h'_{reg}$. Let $\nabla_1$, $\nabla_2$ be two integrable
connections on $E$ with regular singularities. If there exists an
isomorphism $\hat{\psi}:(\widehat{E},\nabla_1)\ra
(\widehat{E},\nabla_2)$, then the local systems $E^{\nabla_1}$ and
$E^{\nabla_2}$ are isomorphic.
\end{lemme}
\begin{proof}
Write $\End(E)$ for the sheaf of endomorphisms of $E$. Then
$\End(E)$ is a locally free sheaf over $\h'_{reg}$. The connections
$\nabla_1$, $\nabla_2$ define a connection $\nabla$ on $\End(E)$ as follows,
$$\nabla: \End(E)\ra\End(E),\quad f\mapsto \nabla_2\circ f-f\circ\nabla_1.$$
So the isomorphism $\hat{\psi}$ is a horizontal section of
$(\widehat{\End(E)},\nabla)$. Let $(\End(E)^\nabla)_0$ be the set of
germs of horizontal sections of $(\End(E),\nabla)$ on zero. By the
Comparison theorem \cite[Theorem 6.3.1]{KK} the canonical map
$(\End(E)^\nabla)_0\ra (\widehat{\End(E)})^\nabla$ is bijective.
Hence there exists a holomorphic isomorphism $\psi: (E,\nabla_1)\ra
(E,\nabla_2)$ which maps to $\hat{\psi}$. Now, let $U$ be an open
ball in $\h'_{reg}$ centered at $0$ with radius $\varepsilon$ small
enough such that the holomorphic isomorphism $\psi$ converges in
$U$. Write $E_U$ for the restriction of $E$ to $U$. Then $\psi$
induces an isomorphism of local systems $(E_U)^{\nabla_1}\cong
(E_U)^{\nabla_2}$. Since $\h'_{reg}$ is homotopy equivalent to $U$,
we have $$E^{\nabla_1}\cong E^{\nabla_2}.$$
\end{proof}

\emph{Step 6.} Finally, write $I$ for the inclusion of
$\Proj_c(W,\h)$ into $\calO_c(W,\h)$. By Lemma \ref{standfilt}(1)
any projective object in $\calO_c(W,\h)$ has a standard filtration,
so (\ref{i}) yields an isomorphism of functors
$$\KZ'\circ\Res_b\circ I\ra \imath_\ast\circ\KZ\circ I.$$ Applying Lemma \ref{projiso} to the exact functors $\KZ'\circ\Res_b$
and $\imath_\ast\circ\KZ$ yields that there is an isomorphism of
functors $$\KZ'\circ\Res_b\cong\imath_\ast\circ\KZ.$$
\end{proof}

\subsection{}\label{ss:coriso}

We give some corollaries of Theorem \ref{iso}.
\begin{cor}\label{indiso}
There is an isomorphism of functors
\begin{equation*}
\KZ\circ\Ind_b\cong\coIndh\circ\KZ'.
\end{equation*}
\end{cor}
\begin{proof}
To simplify notation let us write $$\calO=\calO_c(W,\h),
\quad\calO'=\calO_{c'}(W',\overline{\h}), \quad\Hecke=\Hecke_q(W),
\quad\Hecke'=\Hecke_{q'}(W').$$ Recall that the functor $\KZ$ is
represented by a projective object $P_{\KZ}$ in $\calO$. So for any $N\in
\calO'$ we have a morphism of $\Hecke$-modules
\begin{eqnarray}
\KZ\circ\Ind_b(N)&\cong&\Hom_{\calO}(P_{\KZ},\Ind_b(N))\nonumber\\
&\cong&\Hom_{\calO'}(\Res_b(P_{\KZ}),N)\nonumber\\
\quad&\ra&\Hom_{\Hecke'}(\KZ'(\Res_b (P_{\KZ})), \KZ'(N)).\label{a}
\end{eqnarray}
By Theorem \ref{iso} we have
$$\KZ'\circ\Res_b(P_{\KZ})\cong
\sideset{^{\sss\Hecke}}{^W_{W'}}\Res\circ\KZ(P_{\KZ}).$$ Recall from Section \ref{ss:KZ} that the $\Hecke$-module
$\KZ(P_{\KZ})$ is isomorphic to $\Hecke$. So as $\Hecke'$-modules $\KZ'(\Res_b(P_{\KZ}))$ is also isomorphic to $\Hecke$. Therefore the morphism (\ref{a})
rewrites as
\begin{equation}\label{b}
\chi(N):\KZ\circ\Ind_b(N)\ra\Hom_{\Hecke'}(\Hecke,\KZ'(N)).
\end{equation}
It yields a morphism of functors
$$\chi: \KZ\circ\Ind_b\ra \coIndh\circ
\KZ'.$$ Note that if $N$ is a projective object in $\calO'$, then
$\chi(N)$ is an isomorphism by Proposition \ref{KZ}(1). So Lemma \ref{projiso} implies that $\chi$ is an isomorphism of
functors, because
both functors $\KZ\circ\Ind_b$ and $\coIndh\circ \KZ'$ are exact.
\end{proof}

\subsection{}

The following lemma will be useful to us.

\begin{lemme}\label{fullyfaithful}
Let $K$, $L$ be two right exact functors from $\calO_1$ to $\calO_2$,
where $\calO_1$ and $\calO_2$ can be either $\calO_c(W,\h)$ or
$\calO_{c'}(W',\overline{\h})$. Suppose that $K$, $L$ map projective objects to projective ones. Then the vector space homomorphism
\begin{equation}\label{y}
\Hom(K,L)\ra\Hom(\KZ_2\circ K, \KZ_2\circ L),\quad f\mapsto
1_{\KZ_2}f,
\end{equation}
is an isomorphism.
\end{lemme}
Notice that if $K=L$, this is even an isomorphism of rings.

\begin{proof}
Let $\Proj_1$, $\Proj_2$ be respectively the subcategory of projective objects in $\calO_1$, $\calO_2$. Write $\tilde{K}$, $\tilde{L}$ for the functors from $\Proj_1$ to $\Proj_2$ given by the restrictions of $K$, $L$, respectively.
Let $\Hecke_2$ be the Hecke algebra corresponding to $\calO_2$.
Since the functor $\KZ_2$ is fully faithful over $\Proj_2$ by Proposition \ref{KZ}(1), the following
functor
$$\Fct(\Proj_1,\Proj_2)\ra\Fct(\Proj_1,
\Hecke_2\modu)\,,\quad G\mapsto \KZ_2\circ G$$ is also fully
faithful. This yields an isomorphism
$$\Hom(\tilde{K},\tilde{L})\simra\Hom(\KZ_2\circ\tilde{K},\KZ_2\circ\tilde{L}),\quad f\mapsto 1_{\KZ_2}f.$$
Next, by Lemma \ref{projiso} the canonical morphisms
$$\Hom(K,L)\ra\Hom(\tilde{K}, \tilde{L)},\quad\Hom(\KZ_2\circ
K,\KZ_2\circ L)\ra \Hom(\KZ_2\circ\tilde{K},\KZ_2\circ\tilde{L})$$
are isomorphisms. So the map (\ref{y}) is also an isomorphism.
\end{proof}

Let $b(W,W'')$ be a point in $\h$ whose stabilizer is $W''$. Let
$b(W',W'')$ be its image in $\overline{\h}=\h/\h^{W'}$ via the
canonical projection. Write $b(W,W')=b$.
\begin{cor}\label{corcom}
There are isomorphisms of functors
\begin{eqnarray*}
\Res_{b(W',W'')}\circ\Res_{b(W,W')}&\cong&\Res_{b(W,W'')},\\
\Ind_{b(W,W')}\circ\Ind_{b(W',W'')}&\cong&\Ind_{b(W,W'')}.
\end{eqnarray*}
\end{cor}
\begin{proof}
Since the restriction functors map projective objects to projective ones by Proposition \ref{Res}(1), Lemma
\ref{fullyfaithful} applied to the categories
$\calO_1=\calO_c(W,\h)$, $\calO_2=\calO_{c''}(W'',\h/\h^{W''})$ yields an
isomorphism
\begin{eqnarray*}
&&\Hom(\Res_{b(W',W'')}\circ\Res_{b(W,W')},\Res_{b(W,W'')})\\
&&\cong\Hom(\KZ''\circ\Res_{b(W',W'')}\circ\Res_{b(W,W')},\KZ''\circ\Res_{b(W,W'')}).
\end{eqnarray*}
By Theorem \ref{iso} the set on the second row is
\begin{equation}\label{d}
\Hom(\sideset{^{\sss\Hecke}}{^{W'}_{W''}}\Res\circ\Resh\circ\KZ,
\sideset{^{\sss\Hecke}}{^W_{W''}}\Res\circ\KZ). \end{equation}

By the presentations of Hecke algebras in \cite[Proposition
4.22]{BMR}, there is an isomorphism
$$\sigma:\sideset{^{\sss\Hecke}}{^{W'}_{W''}}\Res\circ\Resh\simra\sideset{^{\sss\Hecke}}{^{W}_{W''}}\Res.$$
Hence the element $\sigma 1_{\KZ}$ in the set (\ref{d}) maps to an
isomorphism
$$\Res_{b(W',W'')}\circ\Res_{b(W,W')}\cong\Res_{b(W,W'')}.$$
This proves the first isomorphism in the corollary. The second one
follows from the uniqueness of right adjoint functor.
\end{proof}

\subsection{Biadjointness of $\Res_b$ and $\Ind_b$.}\label{ss:biadjoint}

Recall that a finite dimensional $\C$-algebra $A$ is
symmetric if $A$ is isomorphic to $A^\ast=\Hom_{\C}(A,\C)$
as $(A,A)$-bimodules.

\begin{lemme}\label{heckeind}
  Assume that $\Hecke_{q}(W)$ and $\Hecke_{q'}(W')$ are symmetric
  algebras. Then the functors $\Indh$ and $\coIndh$ are isomorphic,
  i.e., the functor $\Indh$ is biadjoint to $\Resh$.
\end{lemme}
\begin{proof}
  We abbreviate $\Hecke=\Hecke_{q}(W)$ and $\Hecke'=\Hecke_{q'}(W')$. Since $\Hecke$ is free as a left $\Hecke'$-module,
  for any $\Hecke'$-module $M$ the map
  \begin{equation}\label{eq:proj}
  \Hom_{\Hecke'}(\Hecke,\Hecke')\otimes_{\Hecke'}M\ra
  \Hom_{\Hecke'}(\Hecke,M)
  \end{equation} given by multiplication is an
  isomorphism of $\Hecke$-modules. By assumption $\Hecke'$ is isomorphic to $(\Hecke')^\ast$ as $(\Hecke',\Hecke')$-bimodules.
  Thus we have the following $(\Hecke,\Hecke')$-bimodule isomorphisms
\begin{eqnarray*}
\Hom_{\Hecke'}(\Hecke,\Hecke')&\cong&\Hom_{\Hecke'}(\Hecke,(\Hecke')^\ast)\\
&\cong&\Hom_{\C}(\Hecke'\otimes_{\Hecke'}\Hecke,\C)\\
&\cong&\Hecke^\ast\\
&\cong&\Hecke.
\end{eqnarray*}
The last isomorphism follows from the fact the $\Hecke$ is
symmetric. Thus, by (\ref{eq:proj}) the functors $\Indh$ and $\coIndh$ are isomorphic.
\end{proof}

\begin{rmq}\label{rmq:symmetric}
It is proved that $\Hecke_{q}(W)$ is a symmetric algebra for all
irreducible complex reflection group $W$ except for some of the 34
exceptional groups in the Shephard-Todd classification. See
\cite[Section 2A]{BMM} for details.
\end{rmq}

The biadjointness of $\Res_b$ and $\Ind_b$ was conjectured in
\cite[Remark 3.18]{BE} and was announced by I. Gordon and M. Martino. We give a proof in Proposition \ref{leftadjunction} since it seems not yet to be available in the literature. Let us first consider the following lemma.

\begin{lemme}\label{lem:adjunction}
  (1) Let $A$, $B$ be noetherian algebras and $T$ be a functor $$T:A\modu\ra B\modu.$$ If $T$ is right exact and commutes with direct sums, then it has a right adjoint.

  (2) The functor $$\Res_b:\calO_{c}(W,\h)\ra\calO_{c'}(W',\overline{\h})$$
  has a left adjoint.
\end{lemme}
\begin{proof}
(1) Consider the $(B,A)$-bimodule $M=T(A)$. We claim that the functor $T$ is isomorphic to the functor $M\otimes_A-$. Indeed, by definition we have $T(A)\cong M\otimes_AA$ as $B$ modules. Now, for any $N\in A\modu$, since $N$ is finitely generated and $A$ is noetherian there exists $m$, $n\in\N$ and an exact sequence
$$A^{\oplus n}\lra A^{\oplus m}\lra N\lra 0.$$
Since both $T$ and $M\otimes_A-$ are right exact and they commute with direct sums, the fact that $T(A)\cong M\otimes_AA$ implies that $T(N)\cong M\otimes_AN$ as $B$-modules. This proved the claim. Now, the functor $M\otimes_A-$ has a right adjoint $\Hom_B(M,-)$, so $T$ also has a right adjoint.

(2) Recall that for any complex reflection group $W$, a contravariant duality functor
$$(-)\spcheck:\calO_{c}(W,\h)\ra\calO_{c^\dag}(W,\h^\ast)$$
was defined in \cite[Section 4.2]{GGOR}, here $c^\dag:\calS\ra\C$ is another parameter explicitly determined by $c$. Consider the functor $$\Res_b\spcheck=(-)\spcheck\circ\Res_b\circ(-)\spcheck: \calO_{c^\dag}(W,\h^\ast)\ra \calO_{{c'}^\dag}(W',(\overline{\h})^\ast).$$
The category $\calO_{c^\dag}(W,\h^\ast)$ has a projective generator $P$. The algebra $\End_{\calO_{c^\dag}(W,\h^\ast)}(P)^{\op}$ is finite dimensional over $\C$ and by Morita theory we have an equivalence of categories $$\calO_{c^\dag}(W,\h^\ast)\cong
\End_{\calO_{c^\dag}(W,\h^\ast)}(P)^{\op}\modu.$$ Since the functor $\Res_b\spcheck$ is exact and obviously commutes with direct sums, by part (1) it has a right adjoint $\Psi$. Then it follows that $(-)\spcheck\circ\Psi\circ(-)\spcheck$ is left adjoint to $\Res_b$. The lemma is proved.
\end{proof}
\begin{prop}\label{leftadjunction}
Under the assumption of Lemma \ref{heckeind}, the functor
$\Ind_b$ is left adjoint to $\Res_b$.
\end{prop}
\begin{proof}
\emph{Step 1.} We abbreviate $\calO=\calO_c(W,\h)$, $\calO'=\calO_{c'}(W',\overline{\h})$, $\Hecke=\Hecke_q(W)$, $\Hecke'=\Hecke_{q'}(W')$, and write $\Id_{\calO}$, $\Id_{\calO'}$,
$\Id_{\Hecke}$, $\Id_{\Hecke'}$ for the identity functor on the corresponding categories. We also abbreviate $E^{\sss\Hecke}=\Resh$,
$F^{\sss\Hecke}=\Indh$ and $E=\Res_b$. By Lemma \ref{lem:adjunction} the functor $E$ has a left adjoint. We denote it by
$F:\calO'\ra\calO$. Recall the functors $$\KZ:\calO\ra\Hecke\modu,\quad \KZ':\calO'\ra\Hecke'\modu.$$
The goal of this step is to show that there exists an isomorphism of functors
\begin{equation*}
  \KZ\circ F\cong F^{\sss\Hecke}\circ\KZ'.
\end{equation*}
To this end, let $S$, $S'$ be respectively the right adjoints of $\KZ$, $\KZ'$, see Section \ref{ss:KZ}. We will first give an isomorphism of functors $$F^{\sss\Hecke}\cong\KZ\circ F\circ S'.$$ Let $M\in\Hecke'\modu$ and $N\in\Hecke\modu$. Consider the following equalities given by adjunctions
\begin{eqnarray*}
  \Hom_{\Hecke}(\KZ\circ F\circ S'(M), N)&=&\Hom_{\calO}(F\circ S'(M), S(N))\\
  &=&\Hom_{\calO'}(S'(M), E\circ S(N)).
\end{eqnarray*}
The functor $\KZ'$ yields a map
\begin{equation}\label{eq:mapKZ'}
  a(M,N): \Hom_{\calO'}(S'(M), E\circ S(N))\ra\Hom_{\Hecke'}(\KZ'\circ S'(M), \KZ'\circ E\circ S(N)).
\end{equation}
Since the canonical adjunction maps $\KZ'\circ S'\ra \Id_{\Hecke'}$, $\KZ\circ S\ra \Id_{\Hecke}$ are isomorphisms (see Section \ref{ss:KZ}) and since we have an isomorphism of functors $\KZ'\circ E\cong E^{\sss\Hecke}\circ \KZ$ by Theorem \ref{iso}, we get the following equalities
\begin{eqnarray*}
  \Hom_{\Hecke'}(\KZ'\circ S'(M), \KZ'\circ E\circ S(N))&=&\Hom_{\Hecke'}(M, E^{\sss\Hecke}\circ \KZ\circ S(N))\\
  &=&\Hom_{\Hecke'}(M, E^{\sss\Hecke}(N))\\
  &=&\Hom_{\Hecke}(F^{\sss\Hecke}(M), N).
\end{eqnarray*}
In the last equality we used that $F^{\sss\Hecke}$ is left adjoint to $E^{\sss\Hecke}$. So the map (\ref{eq:mapKZ'}) can be rewritten into the following form
\begin{equation*}
  a(M,N): \Hom_{\Hecke}(\KZ\circ F\circ S'(M), N)\ra\Hom_{\Hecke}(F^{\sss\Hecke}(M), N).
\end{equation*}
Now, take $N=\Hecke$. Recall that $\Hecke$ is isomorphic to $\KZ(P_{\KZ})$ as $\Hecke$-modules. Since $P_{\KZ}$ is projective, by Proposition \ref{KZ}(2) we have a canonical isomorphism in $\calO$ $$P_{\KZ}\cong S(\KZ(P_{\KZ}))=S(\Hecke).$$ Further $E$ maps projectives to projectives by Proposition \ref{Res}(1), so $E\circ S(\Hecke)$ is also projective. Hence Proposition \ref{KZ}(1) implies that in this case (\ref{eq:mapKZ'}) is an isomorphism for any $M$, i.e., we get an isomorphism
\begin{equation*}
 a(M,\Hecke): \Hom_{\Hecke}(\KZ\circ F\circ S'(M), \Hecke)\simra\Hom_{\Hecke}(F^{\sss\Hecke}(M), \Hecke).
\end{equation*}
Further this is an isomorphism of right $\Hecke$-modules with respect to the $\Hecke$-actions induced by the right action of $\Hecke$ on itself. Now, the fact that $\Hecke$ is a symmetric algebra yields that for any finite dimensional $\Hecke$-module $N$ we have isomorphisms of right $\Hecke$-modules
\begin{eqnarray*}
  \Hom_{\Hecke}(N,\Hecke)&\cong&\Hom_{\Hecke}(N,\Hom_{\C}(\Hecke,\C))\\
  &\cong&\Hom_\C(N,\C).
\end{eqnarray*}
Therefore $a(M,\Hecke)$ yields an isomorphism of right $\Hecke$-modules
$$\Hom_{\C}(\KZ\circ F\circ S'(M), \C)\ra\Hom_{\C}(F^{\sss\Hecke}(M), \C).$$
We deduce a natural isomorphism of left $\Hecke$-modules
$$\KZ\circ F\circ S'(M)\cong F^{\sss\Hecke}(M)$$
for any $\Hecke'$-module $M$. This gives an isomorphism of functors $$\psi:\KZ\circ F\circ S'\simra F^{\sss\Hecke}.$$
Finally, consider the canonical adjunction map $\eta:\Id_{\calO'}\ra S'\circ\KZ'$. We have a morphism of functors
$$\phi=(1_{\KZ\circ F}\eta)\circ(\psi 1_{\KZ'}):\KZ\circ F\ra F^{\sss\Hecke}\circ\KZ'.$$
Note that $\psi 1_{\KZ'}$ is an isomorphism of functors. If $Q$ is a projective object in $\calO'$, then by Proposition \ref{KZ}(2) the morphism $\eta(Q): Q\ra S'\circ\KZ'(Q)$ is also an isomorphism, so $\phi(Q)$ is an isomorphism. This implies that $\phi$ is an isomorphism of functors by Lemma \ref{projiso}, because both $\KZ\circ F$ and $F^{\sss\Hecke}\circ\KZ'$ are right exact functors. Here the right exactness of $F$ follows from that it is left adjoint to $E$. So we get the desired isomorphism of functors
$$\KZ\circ F\cong F^{\sss\Hecke}\circ\KZ'.$$

\emph{Step 2.} Let us now prove that $F$ is right adjoint to $E$. By uniqueness of adjoint functors, this will imply that $F$ is isomorphic to $\Ind_b$. First, by Lemma \ref{heckeind} the functor $F^{\sss\Hecke}$ is isomorphic to
$\coIndh$. So $F^{\sss\Hecke}$ is right
adjoint to $E^{\sss\Hecke}$, i.e., we have morphisms of functors
$$\varepsilon^{\sss\Hecke}: E^{\sss\Hecke}\circ F^{\sss\Hecke}\ra\Id_{\Hecke'},\quad
\eta^{\sss\Hecke}: \Id_{\Hecke}\ra F^{\sss\Hecke}\circ E^{\sss\Hecke}$$ such that
$$(\varepsilon^{\sss\Hecke}
1_{E^{\sss\Hecke}})\circ(1_{E^{\sss\Hecke}}\eta^{\sss\Hecke})=1_{E^{\sss\Hecke}},\quad
(1_{F^{\sss\Hecke}}\varepsilon^{\sss\Hecke}
)\circ(\eta^{\sss\Hecke}1_{F^{\sss\Hecke}})=1_{F^{\sss\Hecke}}.$$ Next, both $F$ and $E$ have exact right adjoints, given respectively by $E$ and $\Ind_b$. Therefore $F$ and $E$ map projective objects to projective ones. Applying
Lemma \ref{fullyfaithful} to
$\calO_1=\calO_2=\calO'$, $K=E\circ F$, $L=\Id_{\calO'}$ yields that the following map is bijective
\begin{equation}\label{eq:isoFE}
\Hom(E\circ F,\Id_{\calO'})\ra\Hom(\KZ'\circ E\circ
F,\KZ'\circ\Id_{\calO}),\quad f\mapsto 1_{\KZ'}f. \end{equation} By Theorem \ref{iso} and Step $1$ there
exist isomorphisms of functors $$\phi_E:
E^{\sss\Hecke}\circ\KZ\simra\KZ'\circ E,\quad \phi_F:
F^{\sss\Hecke}\circ\KZ'\simra\KZ\circ F.$$ Let
\begin{eqnarray*}
\phi_{EF}=(\phi_E 1_F)\circ(1_{E^{\sss\Hecke}}\phi_F):
E^{\sss\Hecke}\circ F^{\sss\Hecke}\circ\KZ'\simra\KZ'\circ E\circ F,\\
\phi_{FE}=(\phi_F 1_E)\circ(1_{F^{\sss\Hecke}}\phi_E):F^{\sss\Hecke}\circ
E^{\sss\Hecke}\circ\KZ\simra\KZ\circ F\circ E.
\end{eqnarray*} Identify
$$\KZ\circ\Id_{\calO}=\Id_{\Hecke}\circ\KZ,\quad
\KZ'\circ\Id_{\calO'}=\Id_{\Hecke'}\circ\KZ'.$$ We have a bijective
map
$$\Hom(\KZ'\circ E\circ F,\KZ'\circ\Id_{\calO'})\simra \Hom(E^{\sss\Hecke}\circ
F^{\sss\Hecke}\circ\KZ',\Id_{\Hecke'}\circ\KZ'),\quad g\mapsto
g\circ\phi_{EF}.$$ Together with (\ref{eq:isoFE}), it implies that
there exists a unique morphism $\vep: E\circ F\ra\Id_{\calO'}$ such that
$$(1_{\KZ'}\vep)\circ\phi_{EF}=\vep^{\sss\Hecke}1_{\KZ'}.$$ Similarly,
there exists a unique morphism $\eta: \Id_{\calO}\ra F\circ E$ such that
$$(\phi_{FE})^{-1}\circ(1_{\KZ}\eta)=\eta^{\sss\Hecke}1_{\KZ}.$$
Now, we have the following commutative diagram
$$\xymatrix{E^{\sss\Hecke}\circ\KZ\ar@{=}[r]\ar[d]_{1_{E^{\sss\Hecke}}
\eta^{\sss\Hecke}1_{\KZ}}&E^{\sss\Hecke}\circ\KZ\ar[r]^{\phi_E}
\ar[d]^{1_{E^{\sss\Hecke}}1_{\KZ}\eta}
&\KZ'\circ E\ar[d]^{1_{\KZ'}1_E\eta}\\
E^{\sss\Hecke}\circ F^{\sss\Hecke}\circ
E^{\sss\Hecke}\circ\KZ\quad\ar[r]^{\,1_{E^{\sss\Hecke}}\phi_{FE}}\ar@{=}[d]
&\quad E^{\sss\Hecke}\circ\KZ\circ F\circ E\quad\ar[r]^{\,\phi_E1_F1_E}
&\quad\KZ'\circ
E\circ F\circ E\ar@{=}[d]\\
E^{\sss\Hecke}\circ F^{\sss\Hecke}\circ
E^{\sss\Hecke}\circ\KZ\quad\ar[r]^{\,1_{E^{\sss\Hecke}}1_{F^{\sss\Hecke}}\phi_E}
\ar[d]_{\vep^{\sss\Hecke}1_{E^{\sss\Hecke}}1_{\KZ}}
&\quad E^{\sss\Hecke}\circ F^{\sss\Hecke}\circ\KZ'\circ
E\quad\ar[u]_{1_{E^{\sss\Hecke}}\phi_F1_E}
\ar[r]^{\,\phi_{EF}1_E}\ar[d]^{\vep^{\sss\Hecke}1_{\KZ'}1_E}
&\quad\KZ'\circ E\circ F\circ E\ar[d]^{1_{\KZ'}\vep 1_E}\\
E^{\sss\Hecke}\circ\KZ\ar[r]^{\phi_E} &\KZ'\circ E\ar@{=}[r]&\KZ'\circ
E.}$$ It yields that
$$(1_{\KZ'}\vep 1_E)\circ(1_{\KZ'}1_E\eta)=
\phi_E\circ(\vep^{\sss\Hecke}1_{E^{\sss\Hecke}}1_{\KZ})\circ(1_{E^{\sss\Hecke}}
\eta^{\sss\Hecke}1_{\KZ})\circ(\phi_E)^{-1}.$$
We deduce that
\begin{eqnarray}
1_{\KZ'}((\vep 1_E)\circ(1_E\eta))&=&\phi_E\circ(1_{E^{\sss\Hecke}}1_{\KZ})
\circ(\phi_E)^{-1}\nonumber\\
&=&1_{\KZ'}1_E.\label{eq:unit}
\end{eqnarray}
By applying Lemma \ref{fullyfaithful} to
$\calO_1=\calO$, $\calO_2=\calO'$,
$K=L=E$, we deduce that the following map is bijective
$$\End(E)\ra\End(\KZ'\circ E),\quad f\mapsto1_{\KZ'}f.$$
Hence (\ref{eq:unit}) implies that $$(\vep 1_E)\circ(1_E\eta)=1_E.$$
Similarly, we have $(1_F\varepsilon)\circ (\eta 1_F)=1_F$. So $E$ is left adjoint to $F$. By uniqueness of adjoint functors this implies that $F$ is isomorphic to $\Ind_b$. Therefore $\Ind_b$ is biadjoint to $\Res_b$.
\end{proof}

\section{Reminders on the Cyclotomic case.}\label{s:cyclotomiccase}

From now on we will concentrate on the cyclotomic rational DAHA's.
We fix some notation in this section.

\subsection{}\label{ss:cyclot1}

Let $l,n$ be positive integers. Write $\varepsilon=\exp(\frac{2\pi
\sqrt{-1}}{l})$. Let $\h=\C^n$, write $\{y_1,\ldots,y_n\}$ for its
standard basis. For $1\leqs i,j,k\leqs n$ with $i,j,k$ distinct, let
$\varepsilon_k$, $s_{ij}$ be the following elements of $GL(\h)$:
$$\varepsilon_k(y_k)=\varepsilon y_k,\quad
\varepsilon_k(y_j)=y_j,\quad s_{ij}(y_i)=y_j,\quad
s_{ij}(y_k)=y_k.$$ Let $B_n(l)$ be the subgroup of $GL(\h)$
generated by $\varepsilon_k$ and $s_{ij}$ for $1\leqs k\leqs n$ and
$1\leqs i<j\leqs n$. It is a complex
reflection group with the set of reflections
$$\mathcal{S}_n=\{\varepsilon_i^p:1\leqs i\leqs n, 1\leqs p \leqs l-1\}\bigsqcup
\{s_{ij}^{(p)}=s_{ij}\varepsilon_i^p\varepsilon_j^{-p}:1\leqs
i<j\leqs n, 1\leqs p\leqs l\}.$$
Note that there is an obvious inclusion $\mathcal{S}_{n-1}\hookrightarrow\mathcal{S}_{n}$. It yields an embedding
\begin{equation}\label{eq:inclusiongroup}
 B_{n-1}(l)\hookrightarrow B_n(l).
\end{equation}
This embedding identifies $B_{n-1}(l)$ with the parabolic subgroup of $B_{n}(l)$ given by the stabilizer of the point $b_n=(0,\ldots,0,1)\in\C^n$.

The cyclotomic rational DAHA is the algebra $H_{c}(B_n(l),\h)$. We
will use another presentation in which we replace the parameter $c$
by an $l$-tuple $\bfh=(h,h_1,\ldots,h_{l-1})$ such that
\begin{equation*}
c_{s^{(p)}_{ij}}=-h, \quad
c_{\varepsilon_p}=\frac{-1}{2}\sum_{p'=1}^{l-1}(\varepsilon^{-pp'}-1)h_{p'}.
\end{equation*}

We will denote $H_c(B_n(l),\h)$ by $H_{\bfh,n}$. The corresponding
category $\calO$ will be denoted by $\calO_{\bfh,n}$. In the rest of
the paper, we will fix the positive integer $l$. We will also fix a
positive integer $e\geqs 2$ and an $l$-tuple of integers
$\mathbf{s}=(s_1,\ldots,s_l)$. \emph{We will always assume that the
parameter $\bfh$ is given by the following formulas\,,}
\begin{equation}\label{assumptionh}
h=\frac{-1}{e},\quad h_p=\frac{s_{p+1}-s_p}{e}-\frac{1}{l},\quad
1\leqs p\leqs l-1\,.
\end{equation}

\medskip

The functor $\KZ(B_n(l),\C^n)$ goes from $\calO_{\bfh,n}$ to the
category of finite dimensional modules of a certain Hecke algebra
$\Hecke_{\mathbf{q},n}$ attached to the group $B_n(l)$. Here the
parameter is $\mathbf{q}=(q,q_1,\ldots, q_l)$ with
\begin{equation*}
q=\exp(2\pi\sqrt{-1}/e),\quad q_p=q^{s_p},\quad 1\leqs p\leqs l.
\end{equation*}
The algebra $\Hecke_{\mathbf{q},n}$ has the following presentation:
\begin{itemize}
\item Generators: $T_0, T_1,\ldots, T_{n-1}$,
\item Relations: \begin{gather*}
(T_0-q_1)\cdots(T_0-q_l)=(T_i+1)(T_i-q)=0,\quad 1\leqs i\leqs n-1, \notag \\
T_0T_1T_0T_1=T_1T_0T_1T_0,\notag \\
T_iT_j=T_jT_i,\quad\text{if }|i-j|>1,\label{pres} \\
T_iT_{i+1}T_i=T_{i+1}T_iT_{i+1},\quad 1\leqs i\leqs n-2. \notag
\end{gather*}
\end{itemize}
The algebra $\Hecke_{\mathbf{q},n}$ satisfies the assumption of
Section \ref{s:KZcommute}, i.e., it has the same dimension as $\C
B_n(l)$.

\subsection{}\label{ss:cyclot2}

For each positive integer $n$, the embedding (\ref{eq:inclusiongroup}) of $B_{n}(l)$ into $B_{n+1}(l)$ yields an embedding of Hecke algebras $$\imath_{\mathbf{q}}:
\Hecke_{\mathbf{q},n}\hookrightarrow \Hecke_{\mathbf{q},n+1},$$ see Section \ref{ss:resHecke}. Under
the presentation above this embedding is given by $$\imath_{\mathbf{q}}(T_i)=T_i,\quad\forall\ 0\leqs i\leqs n-1,$$
see \cite[Proposition 2.29]{BMR}.

We will consider the following restriction and induction functors:
\begin{eqnarray*}
E(n)=\Res_{b_n},\quad
E(n)^{\sss\Hecke}=\sideset{^{\sss\Hecke}}{^{B_{n}(l)}_{B_{n-1}(l)}}\Res,\\
F(n)=\Ind_{b_n},\quad
F(n)^{\sss\Hecke}=\sideset{^{\sss\Hecke}}{^{B_{n}(l)}_{B_{n-1}(l)}}\Ind.
\end{eqnarray*}
The algebra $\Hecke_{\mathbf{q},n}$ is symmetric (see Remark
\ref{rmq:symmetric}). Hence by Lemma \ref{heckeind} we have
$$F(n)^{\sss\Hecke}\cong\sideset{^{\sss\Hecke}}{^{B_{n}(l)}_{B_{n-1}(l)}}\coInd.$$
We will abbreviate
$$\calO_{\bfh,\N}=\bigoplus_{n\in\N}\calO_{\bfh,n},\quad \KZ=\bigoplus_{n\in\N}\KZ(B_n(l),\C^n),\quad \Hecke_{\mathbf{q},\N}\modu=\bigoplus_{n\in\N}\Hecke_{\mathbf{q},n}\modu.$$
So $\KZ$ is the Knizhnik-Zamolodchikov functor from
$\calO_{\bfh,\N}$ to $\Hecke_{\mathbf{q},\N}\modu$. Let

\begin{eqnarray*}
E=\bigoplus_{n\geqs 1}E(n),\quad E^{\sss\Hecke}=\bigoplus_{n\geqs
1}E^{\sss\Hecke}(n),\\ F=\bigoplus_{n\geqs 1}F(n),\quad
F^{\sss\Hecke}=\bigoplus_{n\geqs 1}F^{\sss\Hecke}(n).\end{eqnarray*} So
$(E^{\sss\Hecke},F^{\sss\Hecke})$ is a pair of biadjoint endo-functors of $\Hecke_{\bfq,\N}\modu$, and $(E,F)$ is a pair of biadjoint endo-functors of $\calO_{\bfh,\N}$ by Proposition \ref{leftadjunction}.

\subsection{Fock spaces.}\label{ss: fock}

Recall that an $l$-partition is an $l$-tuple $\lam=(\lam^1,\cdots,
\lam^l)$ with each $\lam^j$ a partition, that is a sequence of
integers $(\lam^j)_1\geqs\cdots\geqs(\lam^j)_k>0$. To any
$l$-partition $\lam=(\lam^1,\ldots,\lam^l)$ we attach the set
\begin{equation*}
\Upsilon_\lam=\{(a,b,j)\in \N\times\N\times(\Z/l\Z):
0<b\leqs(\lam^j)_a\}.
\end{equation*}
Write $|\lam|$ for the number of elements in this set, we say that
$\lam$ is an $l$-partition of $|\lam|$. For $n\in\N$ we denote by
$\mathcal{P}_{n,l}$ the set of $l$-partitions of $n$. For any
$l$-partition $\mu$ such that $\Upsilon_\mu$ contains
$\Upsilon_\lam$, we write $\mu/\lam$ for the complement of
$\Upsilon_\lam$ in $\Upsilon_\mu$. Let $|\mu/\lam|$ be the number of
elements in this set. To each element $(a,b,j)$ in $\Upsilon_\lam$
we attach an element $$\res ((a,b,j))=b-a+s_j\in\Z/e\Z,$$ called the
residue of $(a,b,j)$. Here $s_j$ is the $j$-th component of our
fixed $l$-tuple $\bfs$.

The Fock space with multi-charge $\mathbf{s}$ is the $\C$-vector
space $\mathcal{F}_\mathbf{s}$ spanned by the $l$-partitions, i.e.,
$$\mathcal{F}_\mathbf{s}=\bigoplus_{n\in\N}\bigoplus_{\lam\in\mathcal{P}_{n,l}}\C\lam.$$
It admits an integrable $\widehat{\Lie{sl}}_e$-module structure with
the Chevalley generators acting as follows (cf. \cite{JMMO}): for
any $i\in\Z/e\Z$,
\begin{equation}\label{fockei}
e_i(\lam)=\sum_{|\lam/\mu|=1,\res(\lam/\mu)=i}\mu,\quad
f_i(\lam)=\sum_{|\mu/\lam|=1,\res(\mu/\lam)=i}\mu.
\end{equation}
For each $n\in\Z$ set $\Lambda_n=\Lambda_{\underline{n}}$, where
$\underline{n}$ is the image of $n$ in $\Z/e\Z$ and
$\Lambda_{\underline{n}}$ is the corresponding fundamental weight of
$\sle$. Set
$$\Lambda_{\mathbf{s}}=\Lambda_{\underline{s_1}}+\cdots+\Lambda_{\underline{s_l}}.$$
Each $l$-partition $\lam$ is a weight vector of
$\mathcal{F}_\mathbf{s}$ with weight
\begin{equation}\label{wt}
\wt(\lam)=\Lambda_{\mathbf{s}}-\sum_{i\in\Z/e\Z}n_i\al_i,
\end{equation}
where $n_i$ is the number of elements in the set $\{
(a,b,j)\in\Upsilon_\lam: \res((a,b,j))=i\}$. We will call
$\wt(\lam)$ the weight of $\lam$.

In \cite[Section $6.1.1$]{R} an explicit bijection was given between
the sets $\Irr(B_n(l))$ and $\mathcal{P}_{n,l}$. Using this
bijection we identify these two sets and index the standard and
simple modules in $\calO_{\bfh,\N}$ by $l$-partitions. In
particular, we have an isomorphism of $\C$-vector spaces
\begin{equation}\label{Fock}
\theta:
K(\calO_{\bfh,\N})\overset{\sim}\ra\mathcal{F}_{\mathbf{s}},\quad
[\Delta(\lam)]\mapsto \lam.
\end{equation}

\subsection{}\label{kzspecht}

We end this section by the following lemma. Recall that the functor
$\KZ$ gives a map $K(\calO_{\bfh,n})\ra K(\Hecke_{\mathbf{q},n})$.
For any $l$-partition $\lam$ of $n$ let $S_\lam$ be the
corresponding Specht module in $\Hecke_{\mathbf{q},n}\modu$, see
\cite[Definition $13.22$]{A} for its definition.
\begin{lemme}\label{Specht}
In $K(\Hecke_{\mathbf{q},n})$, we have
$\KZ([\Delta(\lam)])=[S_\lam]$.
\end{lemme}
\begin{proof}
Let $R$ be any commutative ring over $\C$. For any $l$-tuplet
$\bfz=(z,z_1,\ldots,z_{l-1})$ of elements in $R$ one defines the
rational DAHA over $R$ attached to $B_n(l)$ with parameter $\bfz$ in
the same way as before. Denote it by $H_{R,\bfz,n}$. The standard
modules $\Delta_R(\lam)$ are also defined as before. For any
$(l+1)$-tuplet $\bfu=(u,u_1,\ldots, u_l)$ of invertible elements in
$R$ the Hecke algebra $\Hecke_{R,\bfu,n}$ over $R$ attached to
$B_n(l)$ with parameter $\bfu$ is defined by the same presentation
as in Section \ref{ss:cyclot1}. The Specht modules $S_{R,\lam}$ are
also well-defined (see \cite{A}). If $R$ is a field, we will write
$\Irr(\Hecke_{R,\bfu,n})$ for the set of isomorphism classes of
simple $\Hecke_{R,\bfu,n}$-modules.

Now, fix $R$ to be the ring of holomorphic functions of one variable
$\varpi$. We choose $\bfz=(z,z_1,\ldots,z_{l-1})$ to be given by
\begin{equation*}
  z=l\varpi,\quad z_p=(s_{p+1}-s_p)l\varpi+e\varpi,\quad 1\leqs p\leqs
  l-1.
\end{equation*}
Write $x=\exp(-2\pi\sqrt{-1}\varpi)$. Let $\bfu=(u,u_1,\ldots, u_l)$
be given by
\begin{equation*}
  u=x^{l},\quad u_p=\vep^{p-1}x^{s_pl-(p-1)e},\quad 1\leqs p\leqs
  l.
\end{equation*}
By \cite[Theorem 4.12]{BMR} the same definition as in Section
\ref{ss:KZ} yields a well defined $\Hecke_{R,\bfu,n}$-module
$$T_{R}(\lam)=\KZ_{R}(\Delta_{R}(\lam)).$$
It is a free $R$-module of finite rank
and it commutes with the base
change functor by the existence and unicity theorem for linear
differential equations, i.e., for any ring homomorphism $R\ra R'$
over $\C$, we have a canonical isomorphism of
$\Hecke_{R',\bfu,n}$-modules
\begin{equation}\label{eq:basechange}
  T_{R'}(\lam)=\KZ_{R'}(\Delta_{R'}(\lam))\cong
T_{R}(\lam)\otimes_RR'.
\end{equation} In particular, for any ring
homomorphism $a: R\ra\C$. Write $\C_a$ for the vector space $\C$
equipped with the $R$-module structure given by $a$. Let $a(\bfz)$,
$a(\bfu)$ denote the images of $\bfz$, $\bfu$ by $a$. Note that we
have $H_{a(\bfz),n}=H_{R,\bfz,n}\otimes_R\C_a$ and
$\Hecke_{a(\bfu),n}=\Hecke_{R,\bfu,n}\otimes_R\C_a$. Denote the
Knizhnik-Zamolodchikov functor of $H_{a(\bfz),n}$ by $\KZ_{a(\bfz)}$
and the standard module corresponding to $\lam$ by
$\Delta_{a(\bfz)}(\lam)$. Then we have an isomorphism of $\Hecke_{a(\bfu),n}$-modules
\begin{equation*}
  T_{R}(\lam)\otimes_{R}\C_{a}\cong\KZ_{a(\bfz)}(\Delta_{a(\bfz)}(\lam)).
\end{equation*}

Let $K$ be the fraction field of $R$. By \cite[Theorem 2.19]{GGOR} the category
$\calO_{K,\bfz,n}$ is split semisimple. In particular, the standard
modules are simple. We have
$$\{T_K(\lam),\lam\in\calP_{n,l}\}=\Irr(\Hecke_{K,\bfu,n}).$$
The Hecke algebra
$\Hecke_{K,\bfu,n}$ is also split semisimple and we have
$$\{S_{K,\lam},\lam\in\calP_{n,l}\}=\Irr(\Hecke_{K,\bfu,n}),$$
see for example \cite[Corollary 13.9]{A}. Thus there is a bijection
$\varphi: \calP_{n,l}\ra\calP_{n,l}$ such that $T_K(\lam)$ is
isomorphic to $S_{K,\varphi(\lam)}$ for all $\lam$. We claim that
$\varphi$ is identity. To see this, consider the algebra
homomorphism $a_0:R\ra\C$ given by $\varpi\mapsto 0$. Then
$\Hecke_{a_0(\bfu),n}$ is canonically isomorphic to the group
algebra $\C B_n(l)$, thus it is semi-simple. Let $\overline{K}$ be
the algebraic closure of $K$. Let $\overline{R}$ be the integral
closure of $R$ in $\overline{K}$ and fix an extension
$\overline{a}_0$ of $a_0$ to $\overline{R}$. By Tit's deformation
theorem (see for example \cite[Section 68A]{CuR}), there is a
bijection
$$\psi:\Irr(\Hecke_{\overline{K},\bfu,n})\simra\Irr(\Hecke_{a_0(\bfu),n})$$
such that
\begin{equation*}
  \psi(T_{\overline{K}}(\lam))=T_{\overline{R}}(\lam)\otimes_{\overline{R}}\C_{\overline{a}_0},\quad
  \psi(S_{\overline{K},\lam})=S_{\overline{R},\lam}\otimes_{\overline{R}}\C_{\overline{a}_0}.
\end{equation*}
By the definition of Specht modules we have
$S_{\overline{R},\lam}\otimes_{\overline{R}}\C_{\overline{a}_0}\cong\lam$
as $\C B_n(l)$-modules. On the other hand, since $a_0(\bfz)=0$, by
(\ref{eq:basechange}) we have the following isomorphisms
\begin{eqnarray*}
  T_{\overline{R}}(\lam)\otimes_{\overline{R}}\C_{\overline{a}_0}&\cong&T_R(\lam)\otimes_R\C_{a_0}\\
  &\cong&\KZ_0(\Delta_0(\lam))\\
  &=&\lam.
\end{eqnarray*}
So $\psi(T_{\overline{K}}(\lam))=\psi(S_{\overline{K},\lam})$.
Hence we have $T_{\overline{K}}(\lam)\cong S_{\overline{K},\lam}$.
Since $T_{\overline{K}}(\lam)=T_K(\lam)\otimes_K\overline{K}$ is
isomorphic to
$S_{\overline{K},\varphi(\lam)}=S_{K,\varphi(\lam)}\otimes_K\overline{K}$,
we deduce that $\varphi(\lam)=\lam$. The claim is proved.

Finally, let $\mathfrak{m}$ be the maximal ideal of $R$ consisting
of the functions vanishing at $\varpi=-1/el$. Let $\widehat R$ be
the completion of $R$ at $\mathfrak{m}$. It is a discrete valuation
ring with residue field $\C$. Let $a_1:\widehat R\ra
\widehat{R}/\mathfrak{m}\widehat{R}=\C$ be the quotient map. We have
$a_1(\bfz)=\bfh$ and $a_1(\bfu)=\bfq$. Let $\widehat{K}$ be the
fraction field of $\widehat R$. Recall that the decomposition map is
given by
$$d: K(\Hecke_{\widehat{K},\bfu,n})\ra K(\Hecke_{\bfq,n}),\quad [M]\mapsto [L\otimes_{\widehat{R}}\C_{a_1}].$$
Here $L$ is any free $\widehat{R}$-submodule of $M$ such that
$L\otimes_{\widehat{R}}\widehat{K}=M$. The choice of $L$ does not
affect the class $[L\otimes_{\widehat{R}}\C_{a_1}]$ in
$K(\Hecke_{\bfq,n})$. See \cite[Section 13.3]{A} for details on this
map. Now, observe that we have
\begin{eqnarray*}
  &d([S_{\widehat{K},\lam}])= [S_{\widehat{R},\lam}\otimes_{\widehat{R}}\C_{a_1}]=[S_\lam],\\
  &d([T_{\widehat{K}}(\lam)])= [T_{\widehat{R}}(\lam)\otimes_{\widehat{R}}\C_{a_1}]=[\KZ(\Delta(\lam))].
\end{eqnarray*}
Since $\widehat{K}$ is an extension of $K$, by the last paragraph we
have $[S_{\widehat{K},\lam}]=[T_{\widehat{K}}(\lam)]$. We deduce
that $[\KZ(\Delta(\lam))]=[S_\lam]$.
\end{proof}

\section{$i$-Restriction and $i$-Induction}\label{iresiind}

We define in this section the $i$-restriction and $i$-induction
functors for the cyclotomic rational DAHA's. This is done in
parallel with the Hecke algebra case.

\subsection{}\label{ss:ireshecke}

Let us recall the definition of the $i$-restriction and
$i$-induction functors for $\Hecke_{\mathbf{q},n}$. First define the
Jucy-Murphy elements $J_0,\ldots, J_{n-1}$ in
$\Hecke_{\mathbf{q},n}$ by
\begin{equation*}
J_0=T_0,\quad J_i=q^{-1}T_iJ_{i-1}T_i\quad\text{ for }1\leqs i\leqs
n-1.
\end{equation*}
Write $Z(\Hecke_{\mathbf{q},n})$ for the center of
$\Hecke_{\mathbf{q},n}$. For any symmetric polynomial $\sigma$ of
$n$ variables the element $\sigma(J_0,\ldots,J_{n-1})$ belongs to
$Z(\Hecke_{\mathbf{q},n})$ (cf. \cite[Section $13.1$]{A}). In
particular, if $z$ is a formal variable the polynomial
$C_n(z)=\prod_{i=0}^{n-1}(z-J_i)$ in $\Hecke_{\mathbf{q},n}[z]$ has
coefficients in $Z(\Hecke_{\mathbf{q},n})$.

Now, for any $a(z)\in\C(z)$ let $P_{n,a(z)}$ be the exact
endo-functor of the category $\Hecke_{\mathbf{q},n}\modu$ that maps an object $M$ to the generalized eigenspace of $C_n(z)$ in $M$
with the eigenvalue $a(z)$.

For any $i\in\Z/e\Z$ the $i$-restriction functor and $i$-induction
functor
$$E_i(n)^{\sss\Hecke}: \Hecke_{\bfq,n}\modu\ra\Hecke_{\bfq,n-1}\modu,
\quad F_i(n)^{\sss\Hecke}:
\Hecke_{\bfq,n-1}\modu\ra\Hecke_{\bfq,n}\modu$$ are defined as
follows (cf. \cite[Definition 13.33]{A}):
\begin{eqnarray*}
E_i(n)^{\sss\Hecke}=\bigoplus_{a(z)\in\C(z)}P_{n-1,a(z)/(z-q^i)}\circ E(n)^{\sss\Hecke}\circ P_{n,a(z)},\label{e}\\
F_i(n)^{\sss\Hecke}=\bigoplus_{a(z)\in\C(z)}P_{n,a(z)(z-q^i)}\circ
F(n)^{\sss\Hecke}\circ P_{n-1,a(z)}.\label{f}
\end{eqnarray*}
We will write
\begin{equation*}
E^{\sss\Hecke}_i=\bigoplus_{n\geqs 1}E_i(n)^{\sss\Hecke},\quad
F^{\sss\Hecke}_i=\bigoplus_{n\geqs 1}F_i(n)^{\sss\Hecke}.
\end{equation*}
They are endo-functors of $\Hecke_{\bfq,\N}$. For each $\lam\in\mathcal{P}_{n,l}$ set $$a_\lam(z)=\prod_{v
\in\Upsilon_\lam}(z-q^{\res(v)}).$$ We recall some properties of
these functors in the following proposition.
\begin{prop}\label{hv}
(1) The functors $E_i(n)^{\sss\Hecke}$, $F_i(n)^{\sss\Hecke}$ are exact. The
functor $E_i(n)^{\sss\Hecke}$ is biadjoint to
$F_i(n)^{\sss\Hecke}$.

(2) For any $\lam\in\mathcal{P}_{n,l}$ the element $C_n(z)$ has a
unique eigenvalue on the Specht module $S_\lam$. It is equal to
$a_\lam(z)$.

(3) We have
\begin{equation*}
E_i(n)^{\sss\Hecke}([S_\lam])=\sum_{\res(\lam/\mu)=i}[S_\mu],\qquad
F_i(n)^{\sss\Hecke}([S_\lam])=\sum_{\res(\mu/\lam)=i}[S_\mu].
\end{equation*}

(4) We have
\begin{equation*}
E(n)^{\sss\Hecke}=\bigoplus_{i\in\Z/e\Z}E_i(n)^{\sss\Hecke},\quad
F(n)^{\sss\Hecke}=\bigoplus_{i\in\Z/e\Z}F_i(n)^{\sss\Hecke}.
\end{equation*}
\end{prop}
\begin{proof}
 Part $(1)$ is obvious. See \cite[Theorem $13.21$(2)]{A} for $(2)$ and \cite[Lemma $13.37$]{A} for $(3)$. Part $(4)$ follows from (3) and \cite[Lemma $13.32$]{A}.
\end{proof}

\subsection{}\label{iresdaha}

By Lemma \ref{lem:center}(1) we have an algebra isomorphism
$$\gamma: Z(\calO_{\bfh,n})\overset{\sim}\ra
Z(\Hecke_{\mathbf{q},n}).$$ So there are unique elements
$K_1,\ldots,K_n\in Z(\calO_{\bfh,n})$ such that the polynomial
$$D_n(z)=z^n+K_1z^{n-1}+\cdots+K_n$$ maps to $C_n(z)$ by $\gamma$. Since the
elements $K_1,\ldots,K_n$ act on simple modules by scalars and the
category $\calO_{\bfh,n}$ is artinian, every module $M$ in
$\calO_{\bfh,n}$ is a direct sum of generalized eigenspaces of
$D_n(z)$. For $a(z)\in\C(z)$ let $Q_{n,a(z)}$ be the exact
endo-functor of $\calO_{\bfh,n}$ which maps an object $M$ to the generalized eigenspace of $D_n(z)$ in $M$ with the eigenvalue
$a(z)$.
\begin{df}\label{def}
The \emph{$i$-restriction} functor and the \emph{$i$-induction}
functor
$$E_i(n): \calO_{\bfh,n}\ra\calO_{\bfh,n-1},\quad F_i(n): \calO_{\bfh,n-1}\ra\calO_{\bfh,n}$$
are given by
\begin{eqnarray*}
E_i(n)=\bigoplus_{a(z)\in\C(z)}Q_{n-1,a(z)/(z-q^i)}\circ E(n)\circ Q_{n,a(z)},\\
F_i(n)=\bigoplus_{a(z)\in\C(z)}Q_{n,a(z)(z-q^i)}\circ F(n)\circ
Q_{n-1,a(z)}.
\end{eqnarray*}
\end{df}
We will write
\begin{equation}\label{ireso}
E_i=\bigoplus_{n\geqs 1}E_i(n),\quad F_i=\bigoplus_{n\geqs 1}F_i(n).
\end{equation}

We have the following proposition.
\begin{prop}\label{isoi}
For any $i\in\Z/e\Z$ there are isomorphisms of functors
\begin{eqnarray*}
\KZ\circ E_i(n)\cong E^{\sss\Hecke}_i(n)\circ\KZ, \quad \KZ\circ
F_i(n)\cong F^{\sss\Hecke}_i(n)\circ\KZ.
\end{eqnarray*}
\end{prop}
\begin{proof}
Since $\gamma(D_n(z))=C_n(z)$, by Lemma \ref{lem:center}(2) for any
$a(z)\in\C(z)$ we have $$\KZ\circ Q_{n,a(z)}\cong P_{n,
a(z)}\circ\KZ.$$ So the proposition follows from Theorem \ref{iso}
and Corollary \ref{indiso}.
\end{proof}

The next proposition is the DAHA version of Proposition \ref{hv}.

\begin{prop}\label{dv}
(1) The functors $E_i(n)$, $F_i(n)$ are exact. The functor $E_i(n)$
is biadjoint to $F_i(n)$.

(2) For any $\lam\in\mathcal{P}_{n,l}$ the unique eigenvalue of
$D_n(z)$ on the standard module $\Delta(\lam)$ is $a_\lam(z)$.

(3) We have the following equalities
\begin{equation}\label{Pierii}
E_i(n)([\Delta(\lam)])=\sum_{\res(\lam/\mu)=i}[\Delta(\mu)],\qquad
F_i(n)([\Delta(\lam)])=\sum_{\res(\mu/\lam)=i}[\Delta(\mu)].
\end{equation}

(4) We have
\begin{equation*}
E(n)=\bigoplus_{i\in\Z/e\Z}E_i(n),\quad
F(n)=\bigoplus_{i\in\Z/e\Z}F_i(n).
\end{equation*}
\end{prop}
\begin{proof}
(1) This is by construction and by Proposition \ref{leftadjunction}.

(2) Since a standard module is indecomposable, the element $D_n(z)$
has a unique eigenvalue on $\Delta(\lam)$. By Lemma \ref{Specht}
this eigenvalue is the same as the eigenvalue of $C_n(z)$ on
$S_\lam$.

(3) Let us prove the equality for $E_i(n)$. The Pieri rule for the
group $B_n(l)$ together with Proposition \ref{Res}(2) yields
\begin{equation}\label{Pieri}
E(n)([\Delta(\lam)])=\sum_{|\lam/\mu|=1}[\Delta(\mu)],\quad
F(n)([\Delta(\lam)])=\sum_{|\mu/\lam|=1}[\Delta(\mu)].
\end{equation}
So we have
\begin{eqnarray*}
E_i(n)([\Delta(\lam)])&=&\bigoplus_{a(z)\in\C[z]}Q_{n-1,a(z)/(z-q^i)}(E(n)(Q_{n,a(z)}([\Delta(\lam)])))\\
&=&Q_{n-1,a_\lam(z)/(z-q^i)}(E(n)(Q_{n,a_\lam(z)}([\Delta(\lam)])))\\
&=&Q_{n-1,a_\lam(z)/(z-q^i)}(E(n)([\Delta(\lam)]))\\
&=&Q_{n-1,a_\lam(z)/(z-q^i)}(\sum_{|\lam/\mu|=1}[\Delta(\mu)])\\
&=&\sum_{\res(\lam/\mu)=i}[\Delta(\mu)].
\end{eqnarray*}
The last equality follows from the fact that for any $l$-partition
$\mu$ such that $|\lam/\mu|=1$ we have
$a_\lam(z)=a_\mu(z)(z-q^{\res(\lam/\mu)})$. The proof for $F_i(n)$
is similar.

(4) It follows from part (3) and (\ref{Pieri}).
\end{proof}

\begin{cor}\label{rep}
Under the isomorphism $\theta$ in (\ref{Fock}) the operators $E_i$
and $F_i$ on $K(\calO_{\bfh,\N})$ go respectively to the operators
$e_i$ and $f_i$ on $\mathcal{F}_\mathbf{s}$. When $i$ runs over
$\Z/e\Z$ they yield an action of $\widehat{\Lie{sl}}_e$ on
$K(\calO_{\bfh,\N})$ such that $\theta$ is an isomorphism of
$\widehat{\Lie{sl}}_e$-modules.
\end{cor}
\begin{proof}
This is clear from Proposition \ref{dv}$(3)$ and from
(\ref{fockei}).
\end{proof}

\section{$\widehat{\Lie{sl}}_e$-categorification}\label{categorification}

In this section, we construct an
$\widehat{\Lie{sl}}_e$-categorification on the category
$\calO_{\bfh,\N}$ under some mild assumption on the parameter $\bfh$
(Theorem \ref{thm:categorification}).

\subsection{}\label{defcategorification}

Recall that we put $q=\exp(\frac{2\pi\sqrt{-1}}{e})$ and $P$ denotes
the weight lattice. Let $\mathcal{C}$ be a $\C$-linear artinian
abelian category. For any functor $F:\mathcal{C}\ra\mathcal{C}$ and
any $X\in\End(F)$, the generalized eigenspace of $X$ acting on $F$
with eigenvalue $a\in\C$ will be called the $a$-eigenspace of $X$ in
$F$. By \cite[Definition 5.29]{R2} an
\emph{$\widehat{\Lie{sl}}_e$-categorification} on $\mathcal{C}$ is
the data of
\begin{itemize}
\item[(a)] an adjoint pair $(U,V)$ of exact functors
$\mathcal{C}\ra\mathcal{C}$,
\item[(b)] $X\in\End(U)$ and $T\in\End(U^2)$,
\item[(c)] a decomposition $\mathcal{C}=\bigoplus_{\tau\in
P}\mathcal{C}_\tau$.
\end{itemize}
such that, set $U_i$ (resp. $V_i$) to be the $q^i$-eigenspace of $X$
in $U$ (resp. in $V$)\footnote{Here $X$ acts on $V$ via the
isomorphism $\End(U)\cong\End(V)^{op}$ given by adjunction, see
\cite[Section 4.1.2]{CR} for the precise definition.} for $i\in\Z/e\Z$, we have
\begin{itemize}
\item[(1)] $U=\bigoplus_{i\in\Z/e\Z}U_i$,
\item[(2)] the endomorphisms $X$ and $T$ satisfy
\begin{eqnarray}
&(1_{U}T)\circ(T1_{U})\circ(1_{U}T)=(T1_{U})\circ(1_{U}T)\circ(T1_{U}),\nonumber\\
&(T+1_{U^2})\circ(T-q1_{U^2})=0,\label{affineHeckerelation}\\
&T\circ(1_{U}X)\circ T=qX1_{U},\nonumber
\end{eqnarray}
\item[(3)] the action of $e_i=U_i$, $f_i=V_i$ on $K(\mathcal{C})$
with $i$ running over $\Z/e\Z$ gives an integrable representation of
$\widehat{\Lie{sl}}_e$.
\item[(4)] $U_i(\mathcal{C}_\tau)\subset \mathcal{C}_{\tau+\al_i}$ and $V_i(\mathcal{C}_\tau)\subset
\mathcal{C}_{\tau-\al_i}$,
\item[(5)] $V$ is isomorphic to a left adjoint of $U$.
\end{itemize}

\subsection{}\label{ss:XT}

We construct a $\sle$-categorification on $\calO_{\bfh,\N}$ in the
following way. The adjoint pair will be given by $(E,F)$. To
construct the part (b) of the data we need to go back to Hecke
algebras. Following \cite[Section $7.2.2$]{CR} let $X^{\sss\Hecke}$ be the
endomorphism of $E^{\sss\Hecke}$ given on $E^{\sss\Hecke}(n)$ as the
multiplication by the Jucy-Murphy element $J_{n-1}$. Let $T^{\sss\Hecke}$
be the endomorphism of $(E^{\sss\Hecke})^2$ given on $E^{\sss\Hecke}(n)\circ
E^{\sss\Hecke}(n-1)$ as the multiplication by the element $T_{n-1}$ in
$\Hecke_{\mathbf{q},n}$. The endomorphisms $X^{\sss\Hecke}$ and $T^{\sss\Hecke}$
satisfy the relations (\ref{affineHeckerelation}). Moreover the
$q^i$-eigenspace of $X^{\sss\Hecke}$ in $E^{\sss\Hecke}$ and $F^{\sss\Hecke}$ gives
respectively the $i$-restriction functor $E^{\sss\Hecke}_i$ and the
$i$-induction functor $F^{\sss\Hecke}_i$ for any $i\in\Z/e\Z$.

By Theorem \ref{iso} we have an isomorphism $\KZ\circ E\cong
E^{\sss\Hecke}\circ\KZ$. This yields an isomorphism
$$\End(\KZ\circ E)\cong \End(E^{\sss\Hecke}\circ\KZ).$$ By Proposition \ref{standard}(1) the functor $E$ maps projective objects to projective ones, so Lemma
\ref{fullyfaithful} applied to
$\calO_1=\calO_2=\calO_{\bfh,\N}$ and $K=L=E$ yields an isomorphism
$$\End(E)\cong\End(\KZ\circ E).$$ Composing it with the isomorphism
above gives a ring isomorphism
\begin{equation}\label{sigmae}
\sigma_{E}:\End(E)\overset{\sim}\ra\End(E^{\sss\Hecke}\circ\KZ).
\end{equation}
Replacing $E$ by $E^2$ we get another isomorphism
$$\sigma_{E^2}:\End(E^2)\overset{\sim}\ra\End((E^{\sss\Hecke})^2\circ\KZ).$$
The data of $X\in\End(E)$ and $T\in\End(E^2)$ in our
$\widehat{\Lie{sl}}_e$-categorification on $\calO_{\bfh,\N}$ will be
provided by
\begin{equation*}
X=\sigma^{-1}_{E}(X^{\sss\Hecke} 1_{\KZ}),\quad
T=\sigma^{-1}_{E^2}(T^{\sss\Hecke} 1_{\KZ}).
\end{equation*}

Finally, the part (c) of the data will be given by the block
decomposition of the category $\calO_{\bfh,\N}$. Recall from
\cite[Theorem 2.11]{LM} that the block decomposition of the category
$\Hecke_{\mathbf{q},\N}\modu$ yields
$$\Hecke_{\mathbf{q},\N}\modu=\bigoplus_{\tau\in P}(\Hecke_{\mathbf{q},\N}\modu)_\tau,$$
where $(\Hecke_{\mathbf{q},\N}\modu)_\tau$ is the subcategory
generated by the composition factors of the Specht modules $S_\lam$
with $\lam$ running over $l$-partitions of weight $\tau$. By
convention $(\Hecke_{\mathbf{q},\N}\modu)_\tau$ is zero if such
$\lam$ does not exist. By Lemma \ref{lem:center} the functor $\KZ$
induces a bijection between the blocks of the category
$\calO_{\bfh,\N}$ and the blocks of $\Hecke_{\mathbf{q},\N}\modu$.
So the block decomposition of $\calO_{\bfh,\N}$ is
\begin{equation*}
\calO_{\bfh,\N}=\bigoplus_{\tau\in P}(\calO_{\bfh,\N})_\tau,
\end{equation*}
where $(\calO_{\bfh,\N})_\tau$ is the block corresponding to
$(\Hecke_{\mathbf{q},\N}\modu)_\tau$ via $\KZ$.

\subsection{}\label{ss:categorification}
Now we prove the following theorem.
\begin{thm}\label{thm:categorification}
The data of
\begin{itemize}
\item[(a)] the adjoint pair $(E,F)$,
\item[(b)] the endomorphisms $X\in\End(E)$, $T\in\End(E^2)$,
\item[(c)] the decomposition $\calO_{\bfh,\N}=\bigoplus_{\tau\in P}(\calO_{\bfh,\N})_\tau$
\end{itemize}
is a $\sle$-categorification on $\calO_{\bfh,\N}$.
\end{thm}
\begin{proof}
First, we prove that for any $i\in\Z/e\Z$ the $q^i$-generalized
eigenspaces of $X$ in $E$ and $F$ are respectively the
$i$-restriction functor $E_i$ and the $i$-induction functor $F_i$ as
defined in (\ref{ireso}).

Recall from Proposition \ref{hv}(4) and Proposition \ref{dv}(4) that
we have $$E=\bigoplus_{i\in\Z/e\Z}E_i\quad\text{ and }\quad
E^{\sss\Hecke}=\bigoplus_{i\in\Z/e\Z}E^{\sss\Hecke}_i.$$ By the proof of
Proposition \ref{isoi} we see that any isomorphism $$\KZ\circ
E\cong E^{\sss\Hecke}\circ\KZ$$ restricts to an isomorphism $\KZ\circ
E_i\cong E^{\sss\Hecke}_i\circ\KZ$ for each $i\in\Z/e\Z$. So the
isomorphism $\sigma_E$ in (\ref{sigmae}) maps $\Hom(E_i,E_j)$ to
$\Hom(E_i^{\sss\Hecke}\circ\KZ,E_j^{\sss\Hecke}\circ\KZ)$. Write
$$X=\sum_{i,j\in\Z/e\Z}X_{ij},\quad X^{\sss\Hecke}
1_{\KZ}=\sum_{i,j\in\Z/e\Z}(X^{\sss\Hecke} 1_{\KZ})_{ij}$$ with
$X_{ij}\in\Hom(E_i, E_j)$ and $(X^{\sss\Hecke}
1_{\KZ})_{ij}\in\Hom(E_i^{\sss\Hecke}\circ\KZ,E_j^{\sss\Hecke}\circ\KZ)$. We
have $$\sigma_E(X_{ij})=(X^{\sss\Hecke} 1_{\KZ})_{ij}.$$ Since
$E^{\sss\Hecke}_i$ is the $q^i$-eigenspace of $X^{\sss\Hecke}$ in $E^{\sss\Hecke}$, we
have $(X^{\sss\Hecke} 1_{\KZ})_{ij}=0$ for $i\neq j$ and $(X^{\sss\Hecke}
1_{\KZ})_{ii}-q^i$ is nilpotent for any $i\in\Z/e\Z$. Since
$\sigma_{E}$ is an isomorphism of rings, this implies that
$X_{ij}=0$ and $X_{ii}-q^i$ is nilpotent in $\End(E)$. So $E_i$ is
the $q^i$-eigenspace of $X$ in $E$. The fact that $F_i$ is the
$q^i$-eigenspace of $X$ in $F$ follows from adjunction.

Now, let us check the conditions (1)--(5):

(1) It is given by Proposition \ref{dv}(4).

(2) Since $X^{\sss\Hecke}$ and $T^{\sss\Hecke}$ satisfy relations in
(\ref{affineHeckerelation}), the endomorphisms $X$ and $T$ also
satisfy them. Because these relations are preserved by ring
homomorphisms.

(3) It follows from Corollary \ref{rep}.

(4) By the definition of $(\calO_{\bfh,\N})_\tau$ and Lemma
\ref{Specht}, the standard modules in $(\calO_{\bfh,\N})_\tau$ are
all the $\Delta(\lam)$ such that $\wt(\lam)=\tau$. By $(\ref{wt})$
if $\mu$ is an $l$-partition such that $\res(\lam/\mu)=i$ then
$\wt(\mu)=\wt(\lam)+\al_i.$ Now, the result follows from
(\ref{Pierii}).

(5) This is Proposition \ref{leftadjunction}.
\end{proof}

\section{Crystals}\label{s:crystal}
Using the $\sle$-categorification in Theorem
\ref{thm:categorification} we construct a crystal on
$\calO_{\bfh,\N}$ and prove that it coincides with the crystal of
the Fock space $\mathcal{F}_\mathbf{s}$ (Theorem \ref{thm:main}).

\subsection{}\label{defcrystal}

A \emph{crystal} (or more precisely, an
$\widehat{\Lie{sl}}_e$-crystal) is a set $B$ together with maps
$$\wt: B\ra P,\quad \tilde{e}_i, \tilde{f}_i: B\ra B\sqcup
\{0\},\quad \epsilon_i,\varphi_i: B\ra
\mathbb{Z}\sqcup\{-\infty\},$$ such that
\begin{itemize}
\item $\varphi_i(b)=\epsilon_{i}(b)+\pair{\al\spcheck_i,\wt(b)}$,
\item if $\tilde{e}_ib\in B$, then
$\wt(\tilde{e}_ib)=\wt(b)+\al_i,\quad
\epsilon_i(\tilde{e}_ib)=\epsilon_i(b)-1,\quad
\varphi_i(\tilde{e}_ib)=\varphi_i(b)+1$,
\item if $\tilde{f}_ib\in B$, then
$\wt(\tilde{f}_ib)=\wt(b)-\al_i,\quad
\epsilon_i(\tilde{f}_ib)=\epsilon_i(b)+1,\quad
\varphi_i(\tilde{f}_ib)=\varphi_i(b)-1$,
\item let $b, b'\in B$, then $\tilde{f}_ib=b'$ if and only if
$\tilde{e}_ib'=b$,
\item if $\varphi_i(b)=-\infty$, then $\tilde{e}_ib=0$ and
$\tilde{f}_ib=0$.
\end{itemize}

Let $V$ be an integrable $\widehat{\Lie{sl}}_e$-module. For any
nonzero $v\in V$ and any $i\in\Z/e\Z$ we set
$$l_i(v)=\max\{l\in\N:\,e_i^{l}v\neq 0\}.$$ Write $l_i(0)=-\infty$.
For $l\geqs 0$ let $$V_i^{<l}=\{v\in V:\,l_i(v) < l\}.$$ A weight
basis of $V$ is a basis $B$ of $V$ such that each element of $B$ is
a weight vector. Following A. Berenstein and D. Kazhdan (cf.
\cite[Definition 5.30]{BK}), a \emph{perfect basis} of $V$ is a
weight basis $B$ together with maps $\tilde{e}_i,\tilde{f}_i: B\ra
B\sqcup\{0\}$ for $i\in \Z/e\Z$ such that
\begin{itemize}
\item for $b, b'\in B$ we have
$\tilde{f}_ib=b'$ if and only if $\tilde{e}_ib'=b,$
\item we have $\tilde{e}_i(b)\neq 0$ if and only if $e_i(b)\neq 0$,
\item if $e_i(b)\neq 0$ then we have
\begin{equation}\label{perf}
e_i(b)\in\C^\ast\tilde{e}_i(b)+V_i^{<l_i(b)-1}.
\end{equation}
\end{itemize}
We denote it by $(B,\tilde{e}_i, \tilde{f}_i)$. For such a basis let
$\mathrm{wt}(b)$ be the weight of $b$, let $\epsilon_i(b)=l_i(b)$
and let
$$\varphi_i(b)=\epsilon_i(b)+\pair{\al_i\spcheck, \mathrm{wt}(b)}$$
for all $b\in B$. The data
\begin{equation}\label{crystaldata}
(B,\wt,\tilde{e}_i,\tilde{f}_i,\epsilon_i,\varphi_i)
\end{equation} is a crystal. We will
always attach this crystal structure to
$(B,\tilde{e}_i,\tilde{f}_i)$. We call $b\in B$ a primitive element
if $e_i(b)=0$ for all $i\in\Z/e\Z$. Let $B^+$ be the set of
primitive elements in $B$. Let $V^+$ be the vector space spanned by
all the primitive vectors in $V$. The following lemma is \cite[Claim
$5.32$]{BK}.
\begin{lemme}\label{basis}
For any perfect basis $(B,\tilde{e}_i,\tilde{f}_i)$ the set $B^+$ is
a basis of $V^+$.
\end{lemme}
\begin{proof}
By definition we have $B^+\subset V^+$. Given a vector $v\in V^+$,
there exist $\zeta_1,\ldots, \zeta_r\in\C^\ast$ and distinct
elements $b_1,\ldots,b_r\in B$ such that $v=\sum_{j=1}^r\zeta_jb_j$.
For any $i\in\Z/e\Z$ let $l_i=\max\{l_i(b_j):\,1\leqs j\leqs r\}$
and $J=\{j:\,l_i(b_j)=l_i,\,1\leqs j\leqs r\}$. Then by the third
property of perfect basis there exist $\eta_j\in\C^\ast$ for $j\in
J$ and a vector $w\in V^{<l_i-1}$ such that $0=e_i(v)=\sum_{j\in
J}\zeta_j\eta_j\tilde{e}_i(b_j)+w$. For distinct $j, j'\in J$, we
have $b_j\neq b_{j'}$, so $\tilde{e}_i(b_j)$ and
$\tilde{e}_i(b_{j'})$ are different unless they are zero. Moreover,
since $l_i(\tilde{e}_i(b_j))=l_i-1$, the equality yields that
$\tilde{e}_i(b_j)=0$ for all $j\in J$. So $l_i=0$. Hence $b_j\in
B^+$ for $j=1,\ldots,r$.
\end{proof}

\subsection{}\label{ss:perfectbasis}

Given an $\sle$-categorification on a $\C$-linear artinian abelian
category $\calC$ with the adjoint pair of endo-functors $(U,V)$,
$X\in\End(U)$ and $T\in\End(U^2)$, one can construct a perfect basis
of $K(\calC)$ as follows. For $i\in\Z/e\Z$ let $U_i$, $V_i$ be the
$q^i$-eigenspaces of $X$ in $U$ and $V$. By definition, the action
of $X$ restricts to each $U_i$. One can prove that $T$ also
restricts to endomorphism of $(U_i)^2$, see for example the
beginning of Section 7 in \cite{CR}. It follows that the data $(U_i,
V_i, X, T)$ gives an $\Lie{sl}_2$-categorification on $\calC$ in the
sense of \cite[Section 5.21]{CR}. By \cite[Proposition 5.20]{CR}
this implies that for any simple object $L$ in $\calC$, the object
$\mathrm{head}(U_i(L))$ (resp. $\mathrm{soc}(V_iL)$) is simple
unless it is zero.

Let $B_{\calC}$ be the set of isomorphism classes of simple objects
in $\calC$. As part of the data of the $\sle$-categorification, we
have a decomposition $\calC=\oplus_{\tau\in P}\calC_\tau$. For a
simple module $L\in\calC_\tau$, the weight of $[L]$ in $K(\calC)$ is
$\tau$. Hence $B_{\calC}$ is a weight basis of $K(\calC)$. Now for
$i\in\Z/e\Z$ define the maps
\begin{eqnarray*}
\tilde{e}_i:& B_{\calC}\ra B_{\calC}\sqcup\{0\},\quad &[L]\mapsto [\head (U_iL)],\\
\tilde{f}_i:& B_{\calC}\ra B_{\calC}\sqcup\{0\},\quad &[L]\mapsto
[\soc (V_iL)].
\end{eqnarray*}
\begin{prop}\label{perfectbasis}
The data $(B_\calC,\tilde{e}_i,\tilde{f}_i)$ is a perfect basis of
$K(\calC)$.
\end{prop}
\begin{proof}
Fix $i\in \Z/e\Z$. Let us check the conditions in the definition in
order:
\begin{itemize}
\item for two simple modules $L$, $L'\in\calC$, we have
$\tilde{e}_i([L])=[L']$ if and only if
$0\neq\Hom(U_iL,L')=\Hom(L,V_iL'),$ if and only if
$\tilde{f}_i([L'])=[L]$.
\item it follows from the fact that any non trivial module has a non
trivial head.
\item this is \cite[Proposition 5.20(d)]{CR}.
\end{itemize}
\end{proof}

\subsection{}\label{ss:mainresult}

Let $B_{\mathcal{F}_\mathbf{s}}$ be the set of $l$-partitions. In
\cite{JMMO} this set is given a crystal structure. We will call it
the crystal of the Fock space $\mathcal{F}_\mathbf{s}$.

\begin{thm}\label{thm:main}
(1) The set
$$B_{\calO_{\bfh,\N}}=\{[L(\lam)]\in K(\calO_{\bfh,\N}): \lam\in\mathcal{P}_{n,l}, n\in\N\}$$
and the maps
\begin{eqnarray*}
\tilde{e}_i:& B_{\calO_{\bfh,\N}}\ra B_{\calO_{\bfh,\N}}\sqcup\{0\},\quad &[L]\mapsto [\head (E_iL)],\\
\tilde{f}_i:& B_{\calO_{\bfh,\N}}\ra
B_{\calO_{\bfh,\N}}\sqcup\{0\},\quad &[L]\mapsto [\soc (F_iL)].
\end{eqnarray*}
define a crystal structure on $B_{\calO_{\bfh,\N}}$.

(2) The crystal $B_{\calO_{\bfh,\N}}$ given by (1) is isomorphic to
the crystal $B_{\mathcal{F}_\mathbf{s}}$.
\end{thm}

\begin{proof}
(1) Applying Proposition \ref{perfectbasis} to the
$\sle$-categorification in Theorem \ref{thm:categorification} yields
that $(B_{\calO_{\bfh,\N}},\tilde{e}_i,\tilde{f}_i)$ is a perfect
basis. So it defines a crystal structure on $B_{\calO_{\bfh,\N}}$ by
(\ref{crystaldata}).

(2) It is known that $B_{\mathcal{F}_\mathbf{s}}$ is a perfect basis
of $\mathcal{F}_\mathbf{s}$. Identify the $\sle$-modules
$\calF_\mathbf{s}$ and $K(\calO_{\bfh,\N})$. By Lemma \ref{basis}
the set $B_{\mathcal{F}_\mathbf{s}}^+$ and $B_{\calO_{\bfh,\N}}^+$
are two weight bases of $\mathcal{F}_\mathbf{s}^+$. So there is a
bijection $\psi: B_{\mathcal{F}_\mathbf{s}}^+\ra
B_{\calO_{\bfh,\N}}^+$ such that $\wt(b)=\wt(\psi(b))$. Since
$\mathcal{F}_\mathbf{s}$ is a direct sum of highest weight simple
$\widehat{\Lie{sl}}_e$-modules, this bijection extends to an
automorphism $\psi$ of the $\widehat{\Lie{sl}}_e$-module
$\mathcal{F}_\mathbf{s}$. By \cite[Main Theorem $5.37$]{BK} any
automorphism of $\mathcal{F}_\mathbf{s}$ which maps
$B_{\mathcal{F}_\mathbf{s}}^+$ to $B_{\calO_{\bfh,\N}}^+$ induces an
isomorphism of crystals $B_{\mathcal{F}_\mathbf{s}}\cong
B_{\calO_{\bfh,\N}}$.
\end{proof}

\begin{rmq}
One can prove that if $n < e$ then a simple module
$L\in\calO_{\bfh,n}$ has finite dimension over $\C$ if and only if
the class $[L]$ is a primitive element in $B_{\calO_{\bfh,\N}}$. In
the case $n=1$, we have $B_n(l)=\mu_l$, the cyclic group, and the
primitive elements in the crystal $B_{\mathcal{F}_\mathbf{s}}$ have
explicit combinatorial descriptions. This yields another proof of
the classification of finite dimensional simple modules of
$H_\bfh(\mu_l)$, which was first given by W. Crawley-Boevey and M. P. Holland. See type $A$ case of \cite[Theorem $7.4$]{CB}.
\end{rmq}

\section*{Acknowledgements}

I would like to thank my advisor E. Vasserot for suggesting me this problem, for his guidance and his patience. I would also like to thank I. Gordon and M. Martino for enlightening discussions. Independently, they also obtained the result on crystals in their recent joint work. I am grateful to D. Yang and I. Losev for pointing out some errors in the preliminary version of this paper.

\end{document}